\documentclass[12pt,twosided,]{amsart} 
\usepackage{amscd,amssymb}
\usepackage[all]{xy}
\usepackage{graphicx}
\usepackage{hhline}

\newtheorem{theorem}{Theorem}[section]
\newtheorem{corollary}[theorem]{Corollary}
\newtheorem{lemma}[theorem]{Lemma}
\newtheorem{prop}[theorem]{Proposition}

\theoremstyle{definition}
\newtheorem{definition}[theorem]{Definition}
\newtheorem{example}[theorem]{Example}
\newtheorem{remark}[theorem]{Remark}

\DeclareMathAlphabet{\pazocal}{OMS}{zplm}{m}{n}

\def\dot{\mathchar"013A}
\newcommand{\hdot}{{\raise1pt\hbox to0.35em{\Huge $\dot$}}}

\begin{document}
\date{August 31, 2024}

\title[Restricted configuration spaces]
{Restricted configuration spaces}

\author[B.R.~Berceanu]{Barbu Rudolf Berceanu}
\email{barberceanu@yahoo.com}

{\em Dedicated to my colleagues in the Faculty of Mathematics, Bucharest, 1970-1974}

\mbox{ }

\thanks{$^2$ This research was partially supported by Viste arquitectura}



\begin{abstract}
Finitely many hypersurfaces are removed from unordered configuration spaces of $n$ points in 
$\mathbb{C}$ to obtain a fibration over unordered configuration spaces of $n-1$ complex points. 
Fundamental groups of these restricted configuration spaces are computed in small dimensions.
\end{abstract}

\maketitle
\setcounter{tocdepth}{1}\tableofcontents

\section{Introduction and statement of results}\label{sec:intro}

It is well-known that the ordered configuration space 
$$ \mathcal{F}_n=\{(z_1,\ldots,z_n)\in\mathbb{C}^n\mid z_i\neq z_j\mbox{ for all }i\neq j\} $$
and the corresponding unordered configuration space ($\Sigma_n$ is the symmetric group)
$$ \mathcal{C}_n=\mathcal{F}_n/\Sigma_n $$
(we identify $\mathcal{C}_n$ with the space of degree $n$ monic complex polynomials with 
non-zero discriminant) are $K(\pi,1)$ spaces with fundamental groups pure braid group 
$\mathcal{P}_n$ and braid group $\mathcal{B}_n$ (see \cite{A} and \cite{B}). There are canonical 
fibrations $pr:\mathcal{F}_n\rightarrow\mathcal{F}_{n-1} $, 
$(z_1,\ldots,z_n)\mapsto(z_1,\ldots,z_{n-1})$, admitting a section $s$, with fiber 
$\mathbb{C}\setminus (n-1)$, the complex plane minus $n-1$ points (see \cite{FN} and \cite{B}).
We will use standard generators of braid group and pure braid group (see \cite{M}); for
example the generators of $\mathcal{B}_3$ are $x_1,x_2$ and the generators of $\mathcal{P}_3$ 
are $A_{12},A_{13},A_{23}$:
\begin{center}
\begin{picture}(360,80)             
\multiput(-10,18)(0,50){2}{\multiput(-10,0)(25,0){3}{$\bullet$}} \put(7,5){$x_1$}
\multiput(65,18)(0,50){2}{\multiput(0,0)(25,0){3}{$\bullet$}}  \put(85,5){$x_2$}
\multiput(160,18)(0,50){2}{\multiput(0,0)(25,0){3}{$\bullet$}} \put(183,5){$A_{12}$}
\multiput(245,18)(0,50){2}{\multiput(0,0)(25,0){3}{$\bullet$}} \put(263,5){$A_{13}$}
\multiput(330,18)(0,50){2}{\multiput(0,0)(25,0){3}{$\bullet$}} \put(347,5){$A_{23}$}
\multiput(-16,70)(109,0){2}{\line(1,-2){25}}  \put(240,45){\line(5,2){5}}                  
\multiput(33,20)(35,0){2}{\line(0,1){50}}     \put(253,52){\line(5,2){17}} 
\multiput(213,20)(60,0){3}{\line(0,1){50}}    \put(296,70){\line(-5,-2){17}}
\put(248,70){\line(0,-1){24}}                 \put(248,20){\line(0,1){17}}
\multiput(-17,20)(109,0){2}{\line(1,2){11}}   \put(240,45){\line(5,-2){30}} 
\multiput(8,70)(109,0){2}{\line(-1,-2){10}}   \put(297,21){\line(-5,2){20}} 
\multiput(164,20)(194,0){2}{\line(1,1){10}}   \multiput(189,45)(194,0){2}{\line(-1,-1){10}}
\multiput(164,45)(194,0){2}{\line(1,-1){25}}  \multiput(164,70)(194,0){2}{\line(1,-1){25}}
\multiput(164,45)(194,0){2}{\line(1,1){10}}    \multiput(189,70)(194,0){2}{\line(-1,-1){10}}
\end{picture}
\end{center}
 
The regular covering $p:\mathcal{F}_n\rightarrow\mathcal{C}_n$ gives the short exact sequence
$$ 1\rightarrow\mathcal{P}_n\overset{p_*}{\longrightarrow}\mathcal{B}_n\longrightarrow
       \Sigma_n\rightarrow 1. $$

Our aim is to find nice projections between a \emph{restricted} unordered configuration space of 
$n$ complex numbers and unordered configuration space of $n-1$ complex numbers. For $n=2$ 
the map $\{z_1,z_2\}\mapsto \frac{z_1+z_2}{2}$ gives a local trivial fibration 
$\mathcal{C}_2\rightarrow\mathcal{C}_1$.  For $n=3$, if we associate to triangle 
$\{z_1,z_2,z_3\}$ the foci $\{f_1,f_2\}$ of the ellipse touching the sides at their midpoints
(and we have $f_1\neq f_2$ if $\{z_1,z_2,z_3\}$ is not an equilateral triangle), we find
a local trivial fibration. For $n=4$, the map $\mathcal{C}_4\rightarrow\mathcal{C}_3$ 
given by Lagrange resolvent is continuous, but it is not a fibration:
$$ P(X)=(X-z_1)(X-z_2)(X-z_3)(X-z_4)\mapsto R(X)=(X-w_1)(X-w_2)(X-w_3),  $$
here $w_1=(z_1+z_2)(z_3+z_4),w_2=(z_1+z_3)(z_2+z_4),w_3=(z_1+z_4)(z_2+z_3)$ and 
${\bf D}_P={\bf D}_R$, see \cite{R} (${\bf D}_A$ is the discriminant of the polynomial $A(X)$). 

In general, the derivative map $D:P(X)\mapsto P'(X)$ is such a projection, if the degree 
$n$ polynomial $P(X)$ and its derivative have non-zero discriminants. Given a degree $n-1$
polynomial with distinct roots, $Q_{n-1}(X)=n(X-\beta_1)\ldots(X-\beta_{n-1})$, the set of
its primitives having distinct roots is
$$ D^{-1}(Q_{n-1})\cap \mathcal{C}_n=\left\lbrace\int_0^XQ_{n-1}(t)dt-\gamma\mid\gamma\neq
           \int_0^{\beta_k}Q_{n-1}(t)dt,\,k=1,\ldots,n-1\right\rbrace. $$
To obtain homeomorphic fibres, the critical values $\lbrace\int_0^{\beta_k}Q_{n-1}(t)dt,\,k=1,
\ldots,n-1\rbrace$ should be distinct. 
\begin{definition} The \emph{restricted base configuration space} $Q\mathcal{C}_{n-1}$ is 
given by 
$$ \left\lbrace Q_{n-1}(X)=n\prod_{k=1}^{n-1}(X-\beta_k)\mid \beta_k\neq\beta_j,
\int_{\beta_k}^{\beta_j}Q_{n-1}(t)dt\neq 0 \mbox{ for all }k\neq j \right\rbrace. $$

The \emph{restricted configuration space} $R\mathcal{C}_n$ is given by 
$$ \left\lbrace 
\begin{array}{ccc}
P_n(X)=\int_0^XQ_{n-1}(t)dt-\gamma & \mid & Q_{n-1}(X)=n\prod_{k=1}^{n-1}(X-\beta_k), \\
 Q_{n-1}(X)\in Q\mathcal{C}_{n-1}  & \mid & \gamma\neq\int_0^{\beta_k}Q_{n-1}(t)dt
\end{array}\right\rbrace. $$
\end{definition}
A more explicit description of $Q\mathcal{C}_{n-1}$ will be given in Section 2. These 
restricted configuration spaces are connected open dense subspaces 
$R\mathcal{C}_n\subset\mathcal{C}_n$ and also
$Q\mathcal{C}_{n-1}\subset\mathcal{C}_{n-1}$.
\begin {theorem}\label{thm1}
The derivative $D:R\mathcal{C}_n\rightarrow Q\mathcal{C}_{n-1}$ is a locally trivial fibration 
with fiber $\mathbb{C}\setminus(n-1)$.This fibration has a section 
$R\mathcal{C}_n\leftarrow Q\mathcal{C}_{n-1}:I$.
\end{theorem}
The inclusions $R\mathcal{C}_n\subset\mathcal{C}_n$ and 
$Q\mathcal{C}_{n-1}\subset\mathcal{C}_{n-1}$ induce regular coverings $Rp$ and $Qp$
\begin{center}
\begin{picture}(360,85)             
\multiput(110,35)(100,-25){2}{$\hookrightarrow$}     \put(82,35){$R\mathcal{C}_n$}   
\multiput(110,65)(100,-25){2}{$\hookrightarrow$}     \put(82,67){$R\mathcal{F}_n$}
\multiput(95,53)(40,0){2}{$\downarrow$}              \put(133,35){$\mathcal{C}_n$}
\multiput(190,28)(40,0){2}{$\downarrow$}             \put(133,67){$\mathcal{F}_n$}
\multiput(115,43)(100,-25){2}{$j$}                   \put(98,30){\vector(4,-1){70}}    
\multiput(115,75)(100,-25){2}{$\tilde{\textit{\j}}$} \put(78,53){$Rp$}
\put(230,42){$\mathcal{F}_{n-1}$}                    \put(170,28){$Qp$}
\multiput(143,55)(95,-25){2}{$p$}                    \put(230,10){$\mathcal{C}_{n-1}$}
\put(173,42){$Q\mathcal{F}_{n-1}$}                   \put(173,10){$Q\mathcal{C}_{n-1}$}
\put(5,53){$\mathbb{C}\setminus (n-1)$}              \put(45,47){\vector(4,-1){35}}
\put(125,10){$D$}                                    \put(220,50){\vector(4,-1){10}}
\multiput(150,65)(4,-1){16}{$\cdot$}                 \put(180,67){$pr$}
\end{picture}
\end{center}
The diagram, without map $pr$, is commutative, but it is not with $pr$.
The fundamental groups of these spaces are denoted $R\mathcal{P}_n$, $R\mathcal{B}_n$, and
$Q\mathcal{P}_{n-1}$, $Q\mathcal{B}_{n-1}$ respectively. Nothing new happens for $n=1,2$: 
$Q\mathcal{C}_n=\mathcal{C}_n$, $R\mathcal{C}_n=\mathcal{C}_n$ (and 
$Q\mathcal{F}_n=\mathcal{F}_n$, $R\mathcal{F}_n=\mathcal{F}_n$). For $n=2,3,4$ all 
spaces in the above diagram are $K(\pi,1)$ spaces. Their fundamental groups and the
corresponding homomorphisms are described in Section 3 for $n=3$. In Section 4 few of
these groups and homomorphisms are described for $n=4$.
 
We use the notation $F(n)$ and $F\langle x_1,\ldots,x_k\rangle$ for the free group 
with $n$ generators and the free group with generators $x_1,\ldots,x_k$. The Garside 
element in $\mathcal{B}_n$ is $\Delta_n=x_1(x_2x_1)\ldots(x_{n-1}x_{n-2}\ldots x_1)$. 
\begin{theorem}\label{thm31}
a) The groups in the diagram 
\begin{center}
\begin{picture}(360,85)               
\multiput(110,35)(0,30){2}{$\longrightarrow$}        \put(82,35){$R\mathcal{B}_3$}   
\multiput(200,12)(0,2){2}{\line(1,0){23}}            \put(82,67){$R\mathcal{P}_3$}
\multiput(95,53)(40,0){2}{$\downarrow$}              \put(133,35){$\mathcal{B}_3$}
\multiput(188,28)(42,0){2}{$\downarrow$}             \put(133,67){$\mathcal{P}_3$}
\multiput(115,43)(95,-25){2}{$j_*$}                  \put(98,30){\vector(4,-1){70}}    
\multiput(115,75)(95,-25){2}{$\tilde{\textit{\j}_*}$}\put(75,53){$Rp_*$}
\put(230,42){$\mathcal{P}_2$}                        \put(167,28){$Qp_*$}
\multiput(143,55)(95,-25){2}{$p_*$}                  \put(230,10){$\mathcal{B}_2$}
\put(173,42){$Q\mathcal{P}_2$}                       \put(173,10){$Q\mathcal{B}_2$}
\put(15,53){$\pi_1(\mathbb{C}\setminus 2)$}          \put(45,47){\vector(4,-1){35}}
\multiput(200,44)(0,2){2}{\line(1,0){23}}            \put(125,10){$D_*$}
\end{picture}
\end{center}
have presentations 
$$ \begin{array}{l}
R\mathcal{B}_3=F\langle\alpha,\beta\rangle\rtimes F\langle\gamma\rangle,\,
           \gamma\alpha\gamma^{-1}=\beta,\gamma\beta\gamma^{-1}=\alpha,              \\
R\mathcal{P}_3=F\langle s_{\bar{\alpha}\alpha},s_{\bar{\alpha}\gamma}, s_{\bar{\beta}
    \beta},s_{\bar{\beta}\gamma}\rangle\times F\langle s_{\bar{\gamma}\gamma}\rangle. 
\end{array}    $$

b) The homomorphisms in the diagram are given by 
\begin{center}
\begin{tabular}{|c|c|c|c|}
\hline
       & $\alpha$    & $\beta$    & $\gamma$     \\
\hline
$j_*$  & $x_2^{-1}$  & $x_1^{-1}$ & $\Delta_3$   \\
\hline
$D_*$  & $1$         & $1$        & $x_1$        \\
\hline
\end{tabular}
\begin{tabular}{|c|c|c|c|c|c|}
\hline
       & $s_{\bar{\alpha}\alpha}$ & $s_{\bar{\alpha}\gamma}$      
       & $s_{\bar{\beta}\beta}$   & $s_{\bar{\beta}\gamma}$  
       & $s_{\bar{\gamma}\gamma}$                                       \\
\hline
$\tilde{\textit{\j}}_*$ 
       & $A_{23}^{-1}$            & $A_{23}^{-1}A_{13}A_{23}A_{12}$     
       & $A_{12}^{-1}$            & $A_{13}A_{23}$
       & $A_{12}A_{13}A_{23}$                                          \\
\hline
$Rp_*$ & $\alpha^2$               & $\alpha\gamma\alpha^{-1}\beta^{-1}$     
       & $\beta^2$                & $\beta\gamma\beta^{-1}\alpha^{-1}$
       & $\gamma^2$                                                    \\
\hline                 
\end{tabular}
\end{center}   
\end{theorem}
\begin{theorem}\label{thm41}
a) In the diagram 
\begin{center}
\begin{picture}(360,85)               
\multiput(110,35)(93,-25){2}{$\longrightarrow$}        \put(82,35){$R\mathcal{B}_4$}   
\multiput(110,65)(93,-25){2}{$\longrightarrow$}        \put(82,67){$R\mathcal{P}_4$}
\multiput(95,53)(40,0){2}{$\downarrow$}                \put(133,35){$\mathcal{B}_4$}
\multiput(188,28)(38,0){2}{$\downarrow$}               \put(133,67){$\mathcal{P}_4$}
\multiput(115,43)(95,-25){2}{$j_*$}                    \put(98,30){\vector(4,-1){70}}    
\multiput(115,75)(95,-25){2}{$\tilde{\textit{\j}_*}$}  \put(75,53){$Rp_*$}
\put(225,42){$\mathcal{P}_3$}                          \put(167,28){$Qp_*$}
\multiput(143,55)(90,-25){2}{$p$}                      \put(225,10){$\mathcal{B}_3$}
\put(173,42){$Q\mathcal{P}_3$}                         \put(173,10){$Q\mathcal{B}_3$}
\put(15,53){$\pi_1(\mathbb{C}\setminus 3)$}            \put(45,47){\vector(4,-1){35}}
\put(125,10){$D_*$}
\end{picture}
\end{center}
we have the following presentations 
$$ \begin{array}{l}
Q\mathcal{P}_3=F\langle\alpha_1,\alpha_0,\alpha_{-\frac{1}{2}},\alpha_{-1},
                  \alpha_{-2}\rangle\times F\langle\beta\rangle,                      \\
Q\mathcal{B}_3=\langle \gamma_1,\gamma_2 \mid 
            \gamma_1\gamma_2\gamma_1\gamma_2\gamma_1\gamma_2=
            \gamma_2\gamma_1\gamma_2\gamma_1\gamma_2\gamma_1\rangle,                  \\
R\mathcal{B}_3=F( \delta_1,\delta_2,\delta_3)\rtimes F(\Gamma_1,\Gamma_2).  
\end{array}    $$
The group $Q\mathcal{B}_3$ is the Artin group of spherical type with graph $I_2(6)$:
\begin{picture}(40,20)
\put(0,2){$\bullet$}        \put(30,2){$\bullet$}
\put(5,5){\line(1,0){25}}   \put(15,10){$6$}
\end{picture}.

b) The homomorphisms in the diagram are given by 
\begin{center}
\begin{tabular}{|c|c|c|c|c|c|c|}
\hline
                        & $\alpha_1$                & $\alpha_0$      
                        & $\alpha_{-\frac{1}{2}}$   & $\alpha_{-1}$  
                        & $\alpha_{-2}$             & $\beta$              \\
\hline
$\tilde{\textit{\j}}_*$ & $A_{12}$                  & $A_{12}$     
                        & $A_{12}A_{13}$            & $A_{12}A_{13}$
                        & $A_{12}A_{13}A_{23}$      & $A_{12}A_{13}A_{23}$ \\
\hline                             
$Qp_*$       & $\gamma_1^2$          
             & $(\gamma_1\gamma_2)^2\gamma_1^{-1}\gamma_2^{-1}$      
             & $(\gamma_1\gamma_2)^2\gamma_1\gamma_2^{-1}$                
             & $\gamma_1\gamma_2^2\gamma_1$
             & $\gamma_1\gamma_2\gamma_1^2\gamma_2\gamma_1=\Delta_3$                
             & $(\gamma_1\gamma_2)^3=\Delta_3$                             \\
\hline
\end{tabular}
\end{center} 
\begin{center}
\begin{tabular}{|c|c|c|}
\hline
             & $\gamma_1$  &   $\gamma_2$                \\
\hline
$j_*$        & $x_1$       &   $x_2$                      \\
\hline
\end{tabular} $\quad$ and $\quad$
\begin{tabular}{|c|c|c|c|c|c|}
\hline
            & $\delta_1$ & $\delta_2$ & $\delta_3$ & $\Gamma_1$ & $\Gamma_2$ \\
\hline
$j_*$       & $x_1$      & $x_2$      & $x_3$      & $x_2$      & $x_3$      \\
\hline           
$D_*$       & $1$        & $1$        & $1$        & $\gamma_1$ & $\gamma_2$ \\
\hline 
\end{tabular}
\end{center} 
\end{theorem}

\begin{corollary}\label{cor31}
a) The space of non-equilateral triangles $\{z_1,z_2,z_3\}$ (in the complex plane) is a 
$K(F(2)\rtimes F(1),1)$ space.

\noindent b) The complement of the hypersurface
$$ S: (z_1-z_2)(z_1-z_3)(z_2-z_3)(z_1^2+z_2^2+z_3^2-z_1z_2-z_1z_3-z_2z_3)=0   $$
in $\mathbb{C}^3$ is a $K(F(4)\times F(1),1)$ space.
\end{corollary}
\begin{corollary}\label{cor40}
a) The group $Q\mathcal{B}_3$ is torsion free and its center is the cyclic group generated 
by ${\bf \Delta}=\gamma_1\gamma_2\gamma_1\gamma_2\gamma_1\gamma_2$.

\end{corollary}

\begin{corollary}\label{cor41}
The space of triples of complex numbers $\{z_1,z_2,z_3\}$ which are not an arithmetic 
progression is a $K(\pi,1)$ space.
\end{corollary}

Section 5 contains a few comments about spaces $Q\mathcal{F}_5$ and $Q\mathcal{F}_6$. We show
that, in general, the new groups $Q\mathcal{P}_{n-1}$, $R\mathcal{P}_n$, $Q\mathcal{B}_{n-1}$ 
and $R\mathcal{B}_n$ are 'more complicated' than the classical 
$\mathcal{P}_*$ and $\mathcal{B}_*$:
\begin{prop}\label{propS}
All the homomorphisms induced by inclusions
$$ \begin{array}{ll}
      \tilde{\textit{\j}}_*:Q\mathcal{P}_{n-1}\twoheadrightarrow\mathcal{P}_{n-1}, &
       j_*:Q\mathcal{B}_{n-1}\twoheadrightarrow\mathcal{B}_{n-1},                            \\
      \tilde{\textit{\j}}_*:R\mathcal{P}_n\twoheadrightarrow\mathcal{P}_n, &
       j_*:R\mathcal{B}_n\twoheadrightarrow\mathcal{B}_n
\end{array}  $$
are surjective.
\end{prop}
We analyse the real case, where new \emph{restricted configuration spaces} 
$Q\mathcal{F}_{n-1}(\mathbb{R})$ and $R\mathcal{F}_n(\mathbb{R})$ (they do not coincide with 
$\mathbb{R}^{n-1}\cap Q\mathcal{F}_{n-1}$ or $\mathbb{R}^n\cap R\mathcal{F}_n$) 
give a trivial fibration
$$ D:R\mathcal{F}_n(\mathbb{R})\longrightarrow Q\mathcal{F}_{n-1}(\mathbb{R})  $$
with contractible fibres.
In particular, we find polynomials of degree $n\geq 4$ with $n$ real distinct roots 
having no primitive with $n+1$ real distinct roots.

\section{Restricted configuration spaces}

Some computations are necessary to give a precise definition for the restricted configuration
spaces $Q\mathcal{F}_{n-1}$ and $R\mathcal{F}_n$. The map $D$ denotes various restrictions of 
the derivative
$$ D:\mathbb{C}[X]\longrightarrow\mathbb{C}[X].   $$
\begin{prop}\label{prop1}
Take a point $Q_{n-1}(X)=n(X-\beta_1)(X-\beta_2)\ldots(X-\beta_{n-1})$ in 
$\mathcal{C}_{n-1}$ ($n\geq 4$) and its primitive $P_n(X)=\int_0^XQ_{n-1}(t)dt$. Then
$$ P_n(\beta_i)-P_n(\beta_j)=\dfrac{-(\beta_i-\beta_j)^3}{(n-1)(n-2)}\Large[A_{n-3}^{(n)}
      (\beta_i,\beta_j)+\sum_{k=1}^{n-3}A_{n-3-k}^{(n)}(\beta_i,\beta_j)\sigma_k\Large], $$
where $\sigma_{\ast}$ are the elementary symmetric polynomials in variables
$z_1,..,\widehat{z_i},..,\widehat{z_j},..,z_{n-1}$, 
$$ A_{n-3}^{(n)}(\beta_i,\beta_j)=(n-2)\beta_i^{n-3}+2(n-3)\beta_i^{n-4}\beta_j+
    3(n-4)\beta_i^{n-5}\beta_j^2+\ldots+(n-2)\beta_j^{n-3} $$
and the coefficients of the polynomials $\sigma_k$ satisfy the recurrence relation
$$ A_{n-3-k}^{(n)}(\beta_i,\beta_j)=\frac{-n}{n-3}A_{n-3-k}^{(n-1)}(\beta_i,\beta_j),\,
                                                         A_0^{(3)}(\beta_i,\beta_j)=1. $$ 
     
\end{prop}
\begin{proof}
We start with $n=3$ where we have, for $Q_2(X)=3(X-\beta_1)(X-\beta_2)$, 
$$ \int_{\beta_2}^{\beta_1}Q_2(t)dt=P_3(\beta_1)-P_3(\beta_2)=\dfrac{-(\beta_1-\beta_2)^3}{2} $$
and, for $n=4$ and 
$P_4(X)=X^4-\frac{4}{3}(\sum\beta_i)X^3+2(\sum\beta_i\beta_j)X^2-4(\prod\beta_i)X$, we find
$$ \int_{\beta_j}^{\beta_i}Q_3(t)dt=P_4(\beta_i)-P_4(\beta_j)=\dfrac{-(\beta_i-\beta_j)^3}
                                              {3\cdot2}[2(\beta_i+\beta_j)-4\sigma_1]. $$
From 
$$ \begin{array}{l}
P_n(X)=X^n-\frac{n}{n-1}(\beta_i+\beta_j+\sigma_1)X^{n-1}+\frac{n}{n-2}[\beta_i\beta_j+
                                   (\beta_i+\beta_j)\sigma_1+\sigma_2]X^{n-2}+\ldots+    \\
\quad \quad +\frac{(-1)^kn}{n-k}[\beta_i\beta_j\sigma_{k-2}+(\beta_i+\beta_j)\sigma_{k-1}+
     \sigma_k]X^{n-k}+\ldots+(-1)^{n-1}n\beta_i\beta_j\sigma_{n-3}X                                 
\end{array}  $$ 
we obtain ($S_n^k=\beta_i^{n-k}+\beta_i^{n-k-1}\beta_j+\beta_i^{n-k-2}\beta_j^2+\ldots+
\beta_j^{n-k}$ and $S_n^n=1$):
$$ \begin{array}{l}
\dfrac{P_n(\beta_i)-P_n(\beta_j)}{\beta_i-\beta_j}=[S_n^1-\frac{n}{n-1}(\beta_i+\beta_j)S_n^2
                                                   +\frac{n}{n-2}\beta_i\beta_jS_n^3]+   \\
\quad \quad +\sigma_1[-\frac{n}{n-1}S_n^2+\frac{n}{n-2}(\beta_i+\beta_j)S_n^3
                                            -\frac{n}{n-3}\beta_i\beta_jS_n^4]+\ldots+   \\
\quad \quad +\sigma_k[\frac{(-1)^kn}{n-k}S_n^{k+1}+\frac{(-1)^{k+1}n}{n-k-1}(\beta_i+\beta_j)
                     S_n^{k+2}+\frac{(-1)^{k+2}n}{n-k-2}\beta_i\beta_jS_n^{k+3}]+\ldots+  \\                                             
\quad \quad +\sigma_{n-3}[\frac{(-1)^{n-3}n}{3}S_n^{n-2}+\frac{(-1)^{n-2}n}{2}
        (\beta_i+\beta_j)S_n^{n-1}+(-1)^{n-1}n\beta_i\beta_jS_n^n]. 
\end{array} $$                                                  
Clearing numerators in the first bracket, we get the polynomial
$$ R_n(\beta_i,\beta_j)=(2-n)\beta_i^{n-1}+2\beta_i^{n-2}\beta_j+2\beta_i^{n-3}\beta_j^2+
                                    \ldots+2\beta_i\beta_j^{n-2}+(2-n)\beta_j^{n-1}; $$
Horner's method, applied twice, gives the expansion:
$$ \dfrac{R_n(\beta_i,\beta_j)}{(\beta_i-\beta_j)^2}=-[(n-2)\beta_i^{n-3}+2(n-3)\beta_i^{n-4}\beta_j+
    3(n-4)\beta_i^{n-5}\beta_j^2+\ldots+(n-2)\beta_j^{n-3}]. $$
Hence the difference $P_n(\beta_i)-P_n(\beta_j)$ starts with the term
$$  \dfrac{-(\beta_i-\beta_j)^3}{(n-1)(n-2)}A_{n-3}^{(n)}(\beta_i,\beta_j).  $$    
For the induction step, we fix two variables $\beta_i,\beta_j$ ($i,j\in\{1,2,\ldots,n\}$)
and we denote by $\overline{\sigma}_1,\overline{\sigma}_2,\ldots,\overline{\sigma}_{n-2}$ 
the elementary symmetric polynomials in the rest of variables. Comparing the previous
expansion with the following one
$$ \begin{array}{l}
\dfrac{P_{n+1}(\beta_i)-P_{n+1}(\beta_j)}{\beta_i-\beta_j}=[S_{n+1}^1-\frac{n+1}{n}
                   (\beta_i+\beta_j)S_{n+1}^2+\frac{n+1}{n-1}\beta_i\beta_jS_{n+1}^3]+   \\
\quad \quad +\overline{\sigma}_1[-\frac{n+1}{n}S_{n+1}^2+\frac{n+1}{n-1}(\beta_i+\beta_j)
                             S_{n+1}^3-\frac{n+1}{n-2}\beta_i\beta_jS_{n+1}^4]+\ldots+   \\
\quad \quad +\overline{\sigma}_{k+1}[\frac{(-1)^{k+1}(n+1)}{n-k}S_{n+1}^{k+2}+
   \frac{(-1)^{k+2}(n+1)}{n-k-1}(\beta_i+\beta_j)S_{n+1}^{k+3}+\frac{(-1)^{k+3}(n+1)}
                                            {n-k-2}\beta_i\beta_jS_{n+1}^{k+4}]+  \\                                             
\quad \quad +\ldots+\overline{\sigma}_{n-2}[\frac{(-1)^{n-2}(n+1)}{3}S_{n+1}^{n-1}+
   \frac{(-1)^{n-1}(n+1)}{2}(\beta_i+\beta_j)S_{n+1}^n+(-1)^n(n+1)\beta_i\beta_j] 
\end{array} $$  
and using $S_{n+1}^{k+1}=S_n^k$ we obtain the recurrence relation.   
\end{proof}
Now we can redefine the restricted configuration spaces. Let us denote 
$H_{i,j}$ the hyperplane in $\mathbb{C}^{n-1}$ given by the equation $z_i=z_j$ and by $S_{i,j}$ 
the hypersurface in $\mathbb{C}^{n-1}$ given by the homogeneous equation of degree $n-3$: 
$$ A_{n-3}^{(n)}(z_i,z_j)+\sum_{k=1}^{n-3}A_{n-3-k}^{(n)}(z_i,z_j)\sigma_k=0. $$
\begin{definition}\label{redef1}
We define, for $n=2$,  
$$ Q\mathcal{C}_1=Q\mathcal{F}_1=\mathcal{C}_1=\mathbb{C},\,R\mathcal{F}_2=\mathcal{F}_2
                        \mbox{ and } R\mathcal{C}_2=\mathcal{C}_2. $$
\noindent For $n=3$ 
$$ \begin{array}{l}
    Q\mathcal{F}_2=\mathcal{F}_2,\,R\mathcal{F}_3=\{(\alpha_1,\alpha_2,\alpha_3)\in
        \mathcal{F}_3\mid \alpha_1^2+\alpha_2^2+\alpha_3^2\neq
        \alpha_1\alpha_2+\alpha_1\alpha_3+\alpha_2\alpha_3 \},                             \\
Q\mathcal{C}_2=\mathcal{C}_2\mbox{ and }R\mathcal{C}_3=\{\{\alpha_1,\alpha_2,\alpha_3\}\in
        \mathcal{C}_3\mid \alpha_1^2+\alpha_2^2+\alpha_3^2\neq
        \alpha_1\alpha_2+\alpha_1\alpha_3+\alpha_2\alpha_3 \}.  
\end{array}        $$

\noindent For $n\geq 4$ we define
$$ Q\mathcal{F}_{n-1}=\mathbb{C}^{n-1}\setminus \cup_{i\neq j}(H_{i,j}\cup S_{i,j}),\,
        Q\mathcal{C}_{n-1}=Q\mathcal{F}_{n-1}/ 
       \Sigma_{n-1} $$
and 
$$ R\mathcal{C}_n=\mathcal{C}_n\cap D^{-1}(Q\mathcal{C}_{n-1}),\,
      R\mathcal{F}_n=p^{-1}(R\mathcal{C}_n) $$
(here $p$ is the covering map $p:\mathcal{F}_n\rightarrow \mathcal{C}_n$).
\end{definition}  

\begin{lemma}\label{lema1}
Take $\gamma_*^0=(\gamma_1^0,\gamma_2^0,\ldots,\gamma_{n-1}^0),\gamma_*=(\gamma_1,\gamma_2,\ldots,
\gamma_{n-1})\in\mathbb{C}^{n-1}$ and $\varepsilon >0$ such that, for any $i$, 
$|\gamma_i-\gamma_i^0|<\varepsilon$ and also $|\gamma_i^0-\gamma_j^0|>3\varepsilon$ for any $i\neq j$. 
Then there is an homeomorphism $\Phi_{\gamma_*^0,\gamma_*}:\mathbb{C}\rightarrow\mathbb{C}$ with 
the following properties:

a) $\Phi_{\gamma_*^0,\gamma_*}(\gamma_i^0)=\gamma_i$ for any $i=1,2,\ldots,n-1$;

b) $\Phi_{\gamma_*^0,\gamma_*}(z)=z$ if $|z-\gamma_i^0|\geq \varepsilon$ for any $i=1,2,\ldots,n-1$;

c) $\Phi_{\gamma_*^0,\gamma_*}$ is continuous in $\gamma_*$.
\end{lemma}
\begin{proof}
One way to obtain such a map is to transform the segments $[\gamma_i^0,\xi]$ (where 
$|\xi-\gamma_i^0|=\varepsilon$) onto the segments $[\gamma_i,\xi]$. Here is an example of a map with 
this property:
$$  \Phi_{\gamma_*^0,\gamma_*}(z)=\begin{cases}
 z+(\gamma_i-\gamma_i^0)(1-\frac{|z-\gamma_i^0|}{\varepsilon}) 
                        & \mbox{ if }|z-\gamma_i^0|\leq\varepsilon,                    \\
 z                      & \mbox{ if }|z-\gamma_j^0|\geq\varepsilon \mbox{ for any }j.  \\
\end{cases}      $$
It is obvious that $|\Phi_{\gamma_*^0,\gamma_*}(z)-\Phi_{\gamma_*^0,\gamma_*'}(z)|\leq
{\rm max}_i|\gamma_i-\gamma_i'|$. 
\end{proof}
\noindent \emph{Proof of Theorem \ref{thm1}}.
For $n=2$ we have the fibration:
$$ \mathbb{C}\setminus 1\hookrightarrow R\mathcal{C}_2=\mathcal{C}_2\stackrel{D}\longrightarrow
        Q\mathcal{C}_1=\mathbb{C}            $$
where $D(\{z_1,z_2\})=\frac{z_1+z_2}{2}$, with the section $I(z)=\{z+1,z-1\}$. 

If $n=3$ we use the first equality in the proof of Proposition \ref{prop1} and the formula for the 
discriminant of the derivative of polynomial $P_3(X)=\prod_{i=1}^3(X-\alpha_i)$: 
$$ {\bf D}_{P'_3}=4(\sum\alpha_i)^2 -12\sum\alpha_i\alpha_j=4(\sum\alpha_i^2 -\sum\alpha_i\alpha_j). $$
\noindent For $n\geq 4$ consider a point 
$Q^0_{n-1}(X)=n(X-\beta_1^0)\ldots(X-\beta_{n-1}^0)\in Q\mathcal{C}_{n-1}$. 
Choose $\delta_1>0$ such that $|\beta_i^0-\beta_j^0|>3\delta_1$ for any $i\neq j$ and a compact disk $K$
containing all the disks $|z-\beta_i^0|<\delta_1$. The fiber $D^{-1}(Q^0_{n-1})$ is the set
$$ \{ P_n^0(X)-\gamma\,|\,\gamma\neq P_n^0(\beta_i^0)\mbox{ for any }i\},\mbox{ where }
                                                       P_n^0(X)=\int_0^XQ^0_{n-1}(t)dt. $$
Chose $\varepsilon>0$ such that $|P_n^0(\beta_i^0)-P_n^0(\beta_j^0)|>3\varepsilon$ (for any $i\neq j$) 
and $\delta_2>0$ such that, for any two points in $K$ satisfying $|\xi-\zeta|<\delta_2$, we have 
$|P_n^0(\xi)-P_n^0(\zeta)|<\frac{\varepsilon}{2}$. Consider an arbitrary point 
$Q_{n-1}(X)=n(X-\beta_1)\ldots(X-\beta_{n-1})\in Q\mathcal{C}_{n-1}$ and $P_n(X)=\int_0^XQ_{n-1}(t)dt$. 
Chose $\delta_3>0$ such that ${\rm sup}_K|P_n(\xi)-P_n^0(\xi)|<\frac{\varepsilon}{2}$ if 
${\rm max}_i|\beta_i-\beta_i^0|<\delta_3$. Take $\delta={\rm min}(\delta_1,\delta_2,\delta_3)$ 
and the neighbourhood $V$ of $Q^0_{n-1}$ defined by 
$$ V=\{Q_{n-1}(X)=n(X-\beta_1)\ldots(X-\beta_{n-1})\,|\,{\rm max}_i {\rm min}_j
                                                               |\beta_i^0-\beta_j|<\delta\}. $$
Every $\beta_i^0$ has at least one $\beta_j$ at distance $<\delta$ and none of $\beta_j$ could be
close to two distinct $\beta_i^0$ (after a re-indexing we can take $|\beta_i^0-\beta_i|<\delta$).
We have 
$$ |P_n(\beta_i)-P_n^0(\beta_i^0)|\leq|P_n(\beta_i)-P_n^0(\beta_i)|+
                                                    |P_n^0(\beta_i)-P_n^0(\beta_i^0)|<\varepsilon,  $$
and, with $\gamma_*^0=(P_n^0(\beta_1^0),\ldots,P_n^0(\beta_{n-1}^0))$ and 
$\gamma_*=(P_n(\beta_1),\ldots,P_n(\beta_{n-1}))$, we can use Lemma \ref{lema1}; we define 
the trivialization map 
$$ \begin{array}{l} 
\Psi:V\times(\mathbb{C}\setminus\{P_n^0(\beta_i^0)\mid i=1,2,\ldots,n-1\})\rightarrow D^{-1}(V), \\
\Psi(Q_{n-1}(X),z)=\Psi(n(X-\beta_1)\ldots(X-\beta_{n-1}),z)=P_n(X)-\Phi_{\gamma_*^0,\gamma_*}(z). 
\end{array}  $$
It is obvious that 
$$ I(Q_{n-1}(X))=\int_0^XQ_{n-1}(t)dt-\left(1+\sum_{i=1}^{n-1}
                                \left|\int_0^{\beta_i}Q_{n-1}(t)dt\right|\right) $$
($\beta_i$ are the roots of $Q_{n-1}(X)$) gives a section of this fibration. 
\hfill{$\square$}

Gauss-Lukas' theorem implies that that the projection map $D$ is decreasing: if 
$\{\alpha_1,\ldots,\alpha_n\}$ and $\{\beta_1,\ldots,\beta_{n-1}\}$ are the roots of $P(X)$ and $D(P)$, 
then the convex hull of $\{\beta_*\}$ is included in the interior of the the convex hull of 
$\{\alpha_*\}$ (if $\alpha_*$ are collinear, then the segment of $\beta_*$ is included in 
the interior of the segment of $\alpha_*$). 
\begin{example}
$\bf{n=2.}$ In this case there are no 'restrictions': 
\begin{center}
\begin{picture}(360,55)               
\put(75,5){$\mathbb{C}\setminus 1$}                  \put(110,7){$\hookrightarrow$}     
\put(140,5){$R\mathcal{C}_2=\mathcal{C}_2$}          \put(200,10){\vector(1,0){40}}
\put(250,5){$Q\mathcal{C}_1=\mathcal{C}_1,\,$}       \put(220,15){$D$}
\put(140,35){$R\mathcal{F}_2=\mathcal{F}_2$}         \put(160,20){$\downarrow$}           
\multiput(265,28)(2,0){2}{\line(0,-1){10}}           \put(140,20){$Rp$}
\put(250,35){$Q\mathcal{F}_1=\mathcal{F}_1$}         \put(235,38){\vector(1,0){5}}
\multiput(200,35)(3,0){12}{$\cdot$}                  \put(220,45){$\tilde{D}$}
\end{picture}
\end{center}  
the derivative is given by $\{z_1,z_2\}\mapsto \frac{z_1+z_2}{2}$ and all spaces are 
$K(\pi,1)$ spaces. Only in this case there is a lift $\tilde{D}$ of the derivative,
$(z_1,z_2)\mapsto\frac{z_1+z_2}{2}$, and a trivial fibration, here 
$\Psi(z_1,z_2)=(\frac{z_1+z_2}{2},\frac{z_1-z_2}{2})$:
\begin{center}
\begin{picture}(360,60) 
\put(130,45){$R\mathcal{F}_2=\mathcal{F}_2$}        \put(185,47){\vector(1,0){45}}
\put(235,45){$Q\mathcal{F}_1\times\mathbb{C}^*$}    \put(160,40){\vector(1,-1){20}}
\put(183,7){$Q\mathcal{F}_1=\mathcal{F}_1$}         \put(250,40){\vector(-1,-1){20}}
\put(205,50){$\Psi$}   \put(151,23){$\tilde{D}$}    \put(247,23){$pr_1$} 
\end{picture}
\end{center}  
\end{example}


\section{Braids of cubic polynomials}\label{sec3}

$\bf{n=3.}$ The base and the fiber in the fibration
$$ \mathbb{C}\setminus 2 \hookrightarrow R\mathcal{C}_3\rightarrow Q\mathcal{C}_2=
                                                                      \mathcal{C}_2 $$
are $K(\pi,1)$ spaces, therefore $R\mathcal{C}_3$ is also a $K(\pi,1)$ space.  We choose 
the base point in $Q\mathcal{C}_2$ the point $3(X^2-1)$ (or $\{-1,1\}$); in 
$R\mathcal{C}_3$ we choose $X^3-3X$ (or $\{-\sqrt{3},0,\sqrt{3}\}$) as the base point 
and $0$ as the base point in $\mathbb{C}\setminus\{-2,2\}$, the fiber $D^{-1}(\{-1,1\})$. 
We define $a(t),b(t)$ representing the two generators of 
$\pi_1(\mathbb{C}\setminus\{-2,2\})\cong F(2)$ by
$$ \begin{array}{l}
   a(t)=\theta(t)+\eta(t)i=\begin{cases}
                \frac{27}{5}t                 & \mbox{ if }t\in[0,\frac{1}{3}],            \\
                2+\frac{1}{5}e^{3\pi i(1-2t)} & \mbox{ if }t\in[\frac{1}{3},\frac{2}{3}],  \\
                \frac{27}{5}(1-t)             & \mbox{ if }t\in[\frac{2}{3},1],            \\
                           \end{cases}    \\ 
   b(t)=-a(t),   \,\, t\in[0,1].
\end{array}$$
We will use $a(t)\neq \pm 2$, $\vert a(t)\vert<\frac{9}{4}$ and $a(t)=\overline{a(1-t)}$. The
generator of $\pi_1(Q\mathcal{C}_2)\cong F(1)$ is represented by the loop 
$c(t)=3(X^2-e^{2\pi it})$ or, equivalently, $c(t)=\{-e^{\pi it},e^{\pi it}\}$. The images of 
$a(t),b(t)$ in $\pi_1(R\mathcal{C}_3)=R\mathcal{B}_3$ are the polynomials:
$$ \alpha(t)=X^3-3X+a(t),\,\beta(t)=X^3-3X+b(t) $$
and $\gamma(t)=X^3-3e^{2\pi it}X$ is a lift of $c(t)$ in this group. The homotopy exact 
sequence of fibration (with a section) 
$\mathbb{C}\setminus 2 \hookrightarrow R\mathcal{C}_3\rightarrow 
Q\mathcal{C}_2$ 
gives a semi-direct decomposition: 
$$ R\mathcal{B}_3\cong F\langle \alpha,\beta \rangle\rtimes F\langle \gamma\rangle $$
In the following pictures, by convention, small circles $\circ$ stand for the base points and 
bullets $\bullet$ stand for the initial and final points of braids or for the missing points
in $\mathbb{C}\setminus n$.
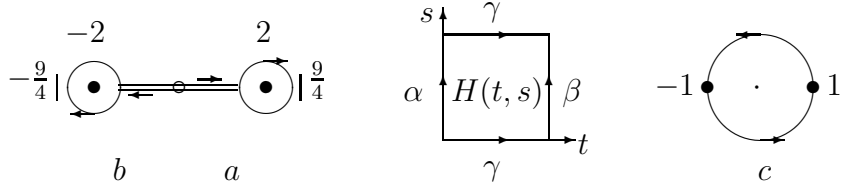
\begin{center}
\begin{picture}(360,70)               
\multiput(48,40)(65,0){2}{\circle{20}}          \put(55,5){$b$}       \put(97,5){$a$}
\multiput(45,37)(65,0){2}{$\bullet $}           \put(37,57){$-2$}     \put(109,57){$2$}
\put(165,35){$\alpha$}                          \put(260,37){$-1$}    \put(325,37){$1$}
\multiput(195,7)(0,60){2}{$\gamma$}             \put(300,40){\circle{40}} 
\put(180,20){\vector(1,0){50}}                  \put(180,60){\line(1,0){40}}
\multiput(200,20)(0,40){2}{\vector(1,0){5}}     \put(180,20){\vector(0,1){50}} 
\multiput(180,40)(40,0){2}{\vector(0,1){5}}     \put(220,20){\line(0,1){40}}
\multiput(277,37)(40,0){2}{$\bullet $}          \put(299,5){$c$}      \put(297,37){$\cdot$}
\put(70,37){\vector(-1,0){9}}                   \put(87,43){\vector(1,0){9}}
\put(300,60){\vector(-1,0){9}}                  \put(300,20){\vector(1,0){9}} 
\put(183,35){$H(t,s)$}                          \put(231,15){$t$}     \put(171,65){$s$}
\multiput(34,35)(92,0){2}{\line(0,1){10}}       \put(15,40){$-\frac{9}{4}$}
\put(129,40){$\frac{9}{4}$}                     \put(225,35){$\beta$}
\multiput(57,39)(0,2){2}{\line(1,0){45}}        \put(112,50){\vector(1,0){9}} 
\put(48,30){\vector(-1,0){9}}                   \put(77,37){$\circ$} 
\end{picture}
\end{center}
\begin{lemma}\label{lema31}
The action of $\gamma$ in $R\mathcal{B}_3$ is given by $\gamma\alpha\gamma^{-1}=\beta$, 
$\gamma\beta\gamma^{-1}=\alpha$. 
\end{lemma}
\begin{proof}
We define the homotopy $H(t,s)=X^3-3e^{2\pi it}\mu(t,s)X+a(s)\nu(t)$, where
$$ \mu(t,s)=\sqrt[3]{1+\frac{(t-t^2)a^2(s)}{4(1-t+t^2)}}\mbox{ and }   
      \nu(t)=\frac{e^{3\pi it}}{\sqrt{1-t+t^2}};  $$ 
the absolute value of $\varrho=\frac{(t-t^2)a^2(s)}{4(1-t+t^2)}$ is less than $\frac{27}{64}$, 
hence Re$(1+\varrho)>0$, so we can choose the (well defined) branch of $\sqrt[3]{1+\varrho}$ 
satisfying $\sqrt[3]{1}=1$. It is clear that
$$ H(0,s)=\alpha(s),\, H(1,s)=-\alpha(s)=\beta(s)\mbox{ and }
H(t,0)=H(t,1)=\gamma(t)   $$
and every polynomial $H(t,s)$ has non-zero discriminant (${\bf D}_{X^3+pX+q}=-4p^3-27q^2$):
$$ \begin{array}{lll}
{\bf D}_{H(t,s)} & = & 4\cdot 27e^{6\pi it}\mu^3(t,s)-27a^2(s)\nu^2(t)=                       \\
                & = & 27e^{6\pi it}[4+\frac{(t-t^2)a^2(s)}{1-t+t^2}-\frac{a^2(s)}{1-t+t^2}]=  \\
                & = & 27e^{6\pi it}[4-a^2(s)]\neq 0.
\end{array}$$
The derivative of $H(t,s)$, $3(X^2-e^{2\pi it}\mu(t,s))$, has a non-zero discriminant, too. 
Therefore $H(t,s)$ gives a homotopy $\alpha\ast\gamma\simeq \gamma\ast\beta$. 
Replacing $a(t)$ with $b(t)$ in the definition of $H(t,s)$ we obtain a homotopy 
$\beta\ast\gamma\simeq\gamma\ast\alpha$.                                  
\end{proof}
\textit{Proof of Theorem \ref{thm31}} a) Reidemeister-Schreier algorithm (see \cite{MKS})
gives a presentation of the normal subgroup $R\mathcal{P}_3\triangleleft R\mathcal{B}_3$: 
we take the Schreier representative system 
$$ [\bar{1},\bar{\alpha},\bar{\beta},\bar{\gamma},\overline{\alpha\beta},
                   \overline{\beta\alpha}]\mapsto [id,(23),(12),(13),(132),(123)].  $$
The images of $\alpha$ and $\beta$ through the homomorphism 
$R\mathcal{B}_3\stackrel{j_*}{\rightarrow}\mathcal{B}_3\rightarrow \Sigma_3$ are the
permutations $(23)$ and $(12)$: a proof is given in part b).                  
From the rewriting system $\tau$ applied to relations 
$\varrho_1=\alpha\gamma\beta^{-1}\gamma^{-1}$, $\varrho_2=\beta\gamma\alpha^{-1}\gamma^{-1}$
and their conjugates, we choose the generators $s_{\bar{\alpha}\alpha}=\alpha^2$, 
$s_{\bar{\alpha}\gamma}=\alpha\gamma\alpha^{-1}\beta^{-1}$, $s_{\bar{\beta}\beta}=\beta^2$, 
$s_{\bar{\beta}\gamma}=\beta\gamma\beta^{-1}\alpha^{-1}$ and $s_{\bar{\gamma}\gamma}=\gamma^2$
and eliminate the others $s_{**}$ (relations $\tau(s_{**})=1$ are not listed):
$$ \begin{array}{lll}
     \tau(\varrho_1)=s_{\bar{\alpha}\gamma}s_{\bar{\gamma}\beta}^{-1} &
     \tau(\varrho_2)=s_{\bar{\beta}\gamma}s_{\bar{\gamma}\alpha}^{-1} &
     \tau(\alpha\varrho_1\alpha^{-1})=s_{\bar{\alpha}\alpha}
     s_{\overline{\beta\alpha}\beta}^{-1}s_{\bar{\alpha}\gamma}^{-1}      \\
     \tau(\alpha\varrho_2\alpha^{-1})=s_{\overline{\alpha\beta}\gamma}
     s_{\overline{\beta\alpha}\alpha}^{-1}s_{\bar{\alpha}\gamma}^{-1} &     
     \tau(\beta\varrho_1\beta^{-1})=s_{\overline{\beta\alpha}\gamma}
     s_{\overline{\alpha\beta}\beta}^{-1}s_{\bar{\beta}\gamma}^{-1}   &
     \tau(\beta\varrho_2\beta^{-1})=s_{\bar{\beta}\beta}
     s_{\overline{\alpha\beta}\alpha}^{-1}s_{\bar{\beta}\gamma}^{-1}      \\
     \tau(\gamma\varrho_1\gamma^{-1})=s_{\bar{\gamma}\alpha}
     s_{\overline{\alpha\beta}\gamma}s_{\bar{\gamma}\gamma}^{-1}      &
     \tau(\gamma\varrho_2\gamma^{-1})=s_{\bar{\gamma}\beta}
     s_{\overline{\beta\alpha}\gamma}s_{\bar{\gamma}\gamma}^{-1}.     &
\end{array} $$
The remaining $\tau(s_{**})$'s give the defining relations of the subgroup:
$$ \begin{array}{lll} 
     \tau(\alpha\beta\varrho_1\beta^{-1}\alpha^{-1})=(s_{\bar{\beta}\gamma}^{-1}
     s_{\bar{\beta}\beta})s_{\bar{\gamma}\gamma}s_{\bar{\beta}\beta}^{-1}
     (s_{\bar{\gamma}\gamma}^{-1}s_{\bar{\beta}\gamma}) &\Rightarrow             & 
     [s_{\bar{\beta}\beta},s_{\bar{\gamma}\gamma}]=1,                              \\
     \tau(\alpha\beta\varrho_2\beta^{-1}\alpha^{-1})=(s_{\bar{\beta}\gamma}^{-1}
     s_{\bar{\alpha}\gamma}^{-1}s_{\bar{\gamma}\gamma})s_{\bar{\alpha}\gamma}
     (s_{\bar{\gamma}\gamma}^{-1}s_{\bar{\beta}\gamma}) &\Rightarrow             & 
     [s_{\bar{\alpha}\gamma},s_{\bar{\gamma}\gamma}]=1,                            \\
     \tau(\beta\alpha\varrho_1\alpha^{-1}\beta^{-1})=(s_{\bar{\alpha}\gamma}^{-1}
     s_{\bar{\beta}\gamma}^{-1}s_{\bar{\gamma}\gamma})s_{\bar{\beta}\gamma}
     (s_{\bar{\gamma}\gamma}^{-1}s_{\bar{\alpha}\gamma}) &\Rightarrow            & 
     [s_{\bar{\beta}\gamma},s_{\bar{\gamma}\gamma}]=1,                              \\
     \tau(\beta\alpha\varrho_2\alpha^{-1}\beta^{-1})=(s_{\bar{\alpha}\gamma}^{-1}
     s_{\bar{\alpha}\alpha})s_{\bar{\gamma}\gamma}s_{\bar{\alpha}\alpha}^{-1}
     (s_{\bar{\gamma}\gamma}^{-1}s_{\bar{\alpha}\gamma}) &\Rightarrow            & 
     [s_{\bar{\alpha}\alpha},s_{\bar{\gamma}\gamma}]=1.
\end{array} $$                                                   
b) In the pictures, by convention, (partial) vertical threads correspond to (parts of)
the paths on the real line and an over-crossing corresponds to a front thread moving in the 
negative half plane Im$(z)<0$ and a back thread moving in the positive half plane Im$(z)>0$.

The values of $D_*$ come from the choice of $\gamma$, a lift of the generator of 
$\pi_1(\mathcal{C}_2)$ The roots of $\gamma(t)=X^3-3e^{2\pi it}X$ are $\{0,\pm\sqrt{3}e^{\pi it}\}$, 
hence $j_*(\gamma)$ is Garside braid $\Delta_3=x_2x_1x_2=x_1x_2x_1$:
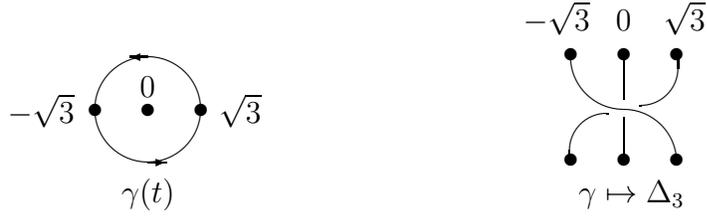
\begin{center}
\begin{picture}(360,80)               
\multiput(57,37)(20,0){3}{$\bullet $}          \put(27,35){$-\sqrt{3}$} 
\put(107,35){$\sqrt{3}$}                       \put(77,45){$0$}
\put(80,40){\oval(40,40)[t]}                   \put(80,40){\oval(40,40)[b]}
\put(80,60){\vector(-1,0){7}}                  \put(80,20){\vector(1,0){7}}
\put(70,5){$\gamma(t)$}                        \put(257,70){$0$} 
\multiput(237,18)(20,0){3}{$\bullet $}         \put(222,70){$-\sqrt{3}$}      
\multiput(237,58)(20,0){3}{$\bullet $}         \put(275,70){$\sqrt{3}$}    
\put(243,5){$\gamma\mapsto \Delta_3$}          \put(260,60){\oval(40,40)[bl]}
\multiput(260,20)(0,24){2}{\line(0,1){17}}     \put(266,60){\oval(29,37)[br]}
\put(260,20){\oval(40,40)[tr]}                 \put(254,20){\oval(29,37)[tl]}
\end{picture}
\end{center}
The roots of $\alpha(t)=X^3-3X+a(t)$ are given by three continuous functions
$$ X_1(t),X_2(t),X_3(t):[0,1]\rightarrow\mathbb{C},\, X_1(0)=-\sqrt{3},X_2(0)=0,
            X_3(0)=\sqrt{3}.  $$
None of these roots intersects the line Re$(z)=-1$ (the real part of $\alpha(-1+\lambda i)$
is $3\lambda^2+2+\theta(t)\geq 2$), hence, for any $t\in[0,1]$, Re$(X_1(t))<-1<{\rm Re}(X_{2,3}(t))$. 
Therefore the thread $X_1(t)$ of the braid $j_*(\alpha(t))$ is
separated from $X_{2,3}(t)$ and $X_1(1)=-\sqrt{3}$, $\{X_2(1),X_3(1)\}=\{0,\sqrt{3}\}$.
Rolle sequence for the real function $\alpha(t)$, $t\in[0,\frac{1}{3}]$, gives real
roots $X_{2,3}(t)$ where $X_2(t)$ covers the interval $[0,\chi_2]$, $\chi_2<\frac{4}{5}$,
and $X_3(t)$ covers $[\chi_3,\sqrt{3}]$, $\chi_3>\frac{6}{5}$. The root $ X_1(t)$ is real
if and only if $t\in[0,\frac{1}{3}]\cup\{\frac{1}{2}\}\cup[\frac{2}{3},1]$ and the roots 
$X_j(t)=U_j+V_j(t)i$ ($j=2,3$), are real if and only if 
$t\in[0,\frac{1}{3}]\cup[\frac{2}{3},1]$. For $t=\frac{1}{2}$, $X_1(\frac{1}{2})\in(-3,-2)$,
$X_2(\frac{1}{2})=\overline{X_3(\frac{1}{2})}\notin\mathbb{R}$, and this implies that, for
$t\in(\frac{1}{3},\frac{2}{3})$, $V_2(t)V_3(t)<0$ and the roots $X_2(t)$, $X_3(t)$ are 
separated by the real line Im$(z)=0$. The roots of $\alpha(t)$ are the conjugates of the
roots of $\alpha(1-t)$, hence
$\lim_{t\to\frac{2}{3},t<\frac{2}{3}}X_2(t)=\lim_{t\to\frac{1}{3},t>\frac{1}{3}}
    \overline{X_3(t)}=\chi_3$,
and we find that $X_2(1)=\sqrt{3}$, $X_3(1)=0$, therefore $j_*(\alpha)=x_2$ or 
$j_*(\alpha)=x_2^{-1}$. To see that the latter is correct, we show that 
${\rm Im}(X_2(t))\geq 0\geq{\rm Im}(X_3(t))$: from 
$$ (U+Vi)^3-3(U+Vi)+\theta+\eta i=0 \mbox{ we get }V=\dfrac{3U\eta}{-8U^3+6U+\theta}. $$
For $t<\frac{1}{3}$, we have $V_2(t)=0$ and 
$$ U_2(t)^3-3U_2(t)+\theta(t)=0, \mbox{ hence }
                 -8U_2(t)^3+6U_2(t)+\theta(t)=9U_2(t)(1-U_2^2(t)). $$
For $t$ near $\frac{1}{3}$, $U_2(t)$ is near $\chi_2<1$, hence $V_2(t)>0$.
The paths $X_{1,2,3}(t)$ and the corresponding braid $j_*(\alpha)$ are given in the picture 
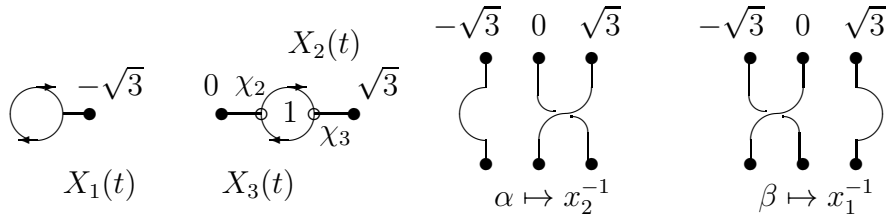
\begin{center}
\begin{picture}(360,80)                   
\multiput(112,37)(20,0){2}{$\circ$} 
\multiput(47,37)(50,0){3}{$\bullet $}         \put(45,47){$-\sqrt{3}$}  
\multiput(30,40)(95,0){2}{\circle{20}}        \put(151,47){$\sqrt{3}$}
\multiput(100,40)(35,0){2}{\line(1,0){15}}    \put(93,47){$0$}
\multiput(340,20)(0,30){2}{\line(0,1){10}}    \put(40,10){$X_1(t)$}
\put(125,63){$X_2(t)$}                        \put(100,10){$X_3(t)$}
\multiput(30,50)(95,0){2}{\vector(1,0){7}}    \put(137,30){$\chi_3$}        
\multiput(30,30)(95,0){2}{\vector(-1,0){7}}   \put(105,49){$\chi_2$}
\put(40,40){\line(1,0){10}}                   \put(230,60){\oval(20,40)[br]}  
\multiput(197,18)(20,0){3}{$\bullet $}        \put(230,20){\oval(20,40)[tl]} 
\multiput(180,70)(100,0){2}{$-\sqrt{3}$}      \put(226,60){\oval(13,38)[bl]} 
\multiput(297,18)(20,0){3}{$\bullet $}        \put(232,20){\oval(13,38)[tr]} 
\multiput(235,70)(100,0){2}{$\sqrt{3}$}       \put(340,40){\oval(20,20)[r]}
\multiput(197,58)(20,0){3}{$\bullet $}        \put(310,60){\oval(20,40)[br]}  
\multiput(297,58)(20,0){3}{$\bullet $}        \put(306,60){\oval(13,38)[bl]} 
\multiput(217,70)(100,0){2}{$0$}              \put(312,20){\oval(13,38)[tr]}
\put(203,5){$\alpha\mapsto x_2^{-1}$}         \put(310,20){\oval(20,40)[tl]} 
\put(303,5){$\beta\mapsto x_1^{-1}$}          \put(200,40){\oval(20,20)[l]}
\multiput(200,20)(0,30){2}{\line(0,1){10}}    \put(123,37){$1$} 
\end{picture}
\end{center}
There is a similar proof for $j_*(\beta)=x_1^{-1}$.  

Using the values of $Rp_*(s_{**})$ from \ref{thm31},  for instance
$$ \begin{array}{lll} 
 p_*\tilde{\textit{\j}}_*(s_{\bar{\alpha}\alpha}) & = & j_*Rp_*(s_{\bar{\alpha}\alpha})=
     j_*(\alpha^2)=x_2^{-2}=p_*(A_{23}^{-1}),                                              \\
 p_*\tilde{\textit{\j}}_*(s_{\bar{\alpha}\gamma}) & = &j_*Rp_*(s_{\bar{\alpha}\gamma})=
     j_*(\alpha\gamma\alpha^{-1}\beta^{-1})=x_2^{-1}\Delta_3x_2x_1=x_2^{-1}x_1\Delta_3x_1= \\       
                                                  & = & x_2^{-1}x_1^2x_2x_1^2=
     x_2^{-2}(x_2x_1^2x_2^{-1})x_2^2x_1^2=p_*(A_{23}^{-1}A_{13}A_{23}A_{12}),              \\           
  p_*\tilde{\textit{\j}}_*(s_{\bar{\beta}\gamma}) & = & j_*Rp_*(s_{\bar{\beta}\gamma})=
     j_*(\beta\gamma\beta^{-1}\alpha^{-1})=x_1^{-1}\Delta_3x_1x_2=                         \\       
                                                  & = & x_2x_1^2x_2=
                       (x_2x_1^2x_2^{-1})x_2^2=p_*(A_{13}A_{23}),                          \\    
  p_*\tilde{\textit{\j}}_*(s_{\bar{\gamma}\gamma})& = & j_*Rp_*(s_{\bar{\gamma}\gamma})=
     j_*(\gamma^2)=(x_1x_2x_1)\Delta_3=                                                    \\ 
                                                  & = & x_1\Delta_3x_1x_2=
     x_1^2(x_2x_1^2x_2^{-1})x_2^2=p_*(A_{12}A_{13}A_{23}),   
\end{array} $$
we find the values of $\tilde{\textit{\j}}_*(s_{**})$.               \hfill $\square$ \\
\textit{Proof of Corollary \ref{cor31}.} a) The roots $\beta_1,\beta_2$ of the derivative of 
the polynomial $P_3(X)=(X-\alpha_1)(X-\alpha_2)(X-\alpha_3)$ are the foci of the ellipse 
touching the sides of triangle $\alpha_1\alpha_2\alpha_3$ at their midpoints (theorem
of van der Berg, see \cite{P}). We have $\beta_1=\beta_2$ if and only if the ellipse is a 
circle and this happens if and only if $\alpha_1\alpha_2\alpha_3$ is an equilateral triangle.

b) This is a consequence of Theorem \ref{thm31} b), because we have 
$$ \alpha_1^2+\alpha_2^2+\alpha_3^2=\alpha_1\alpha_2+\alpha_1\alpha_3+\alpha_2\alpha_3 $$ 
if and only if triangle $\alpha_1\alpha_2\alpha_3$ is equilateral. 

From factorization 
$$ \alpha_1^2+\alpha_2^2+\alpha_3^2-\alpha_1\alpha_2-\alpha_1\alpha_3-\alpha_2\alpha_3=
   (\alpha_1+\omega\alpha_2 +\omega^2\alpha_3)(\alpha_1+\omega^2\alpha_2+\omega\alpha_3), $$
the space $R\mathcal{C}_3$ is a complement of a central arrangements in $\mathbb{C}^3$ 
(see \cite{OT}), therefore  
we have a direct proof of Theorem \ref{thm31} b) and its Corollary.       \hfill $\square$      


\section{Braids of quartic polynomials}\label{sec4}

$\bf{n=4.}$ The computation of fundamental groups in Theorem \ref{thm41} is given in the order
$Q\mathcal{P}_3$, $Q\mathcal{B}_3$, $R\mathcal{B}_4$. 

$\bf{Q\mathcal{P}_3.}$ In $\mathbb{C}^3$ take $\mathcal{A}$ the arrangement of hyperplanes given by
$$ D_{ij}: [X_i=X_j], S_{ij}:[X_i+X_j=2X_k],\mbox{ where }i,j,k \mbox{ are distinct} $$
and also, in $\mathbb{C}^2$, the central arrangement of lines $\widehat{\mathcal{A}}$ given by
$$ [X-Y=0],[2X+Y=0],[X+2Y=0],[X=0],[Y=0],[X+Y=0]. $$
We take $(0,1,3)$ as base point in $\mathbb{C}^3\setminus\mathcal{A}=Q\mathcal{F}_3$ (and
also in $\mathcal{F}_3$). All the hyperplanes of $\mathcal{A}$ contain the diagonal 
$\{(z,z,z)\}$. In fact, this diagonal is the intersection of any two hyperplanes $D_{ij},S_{hk}$.
\begin{lemma}\label{lema40} There are homeomorphisms 
$$ Q\mathcal{F}_3\overset{\Phi}{\longrightarrow}S=(\mathbb{C}^2\setminus
    \widehat{\mathcal{A}})\times\mathbb{C}\overset{\Psi}{\longrightarrow}T=
                 (\mathbb{C}\setminus H)\times\mathbb{C}^*\times\mathbb{C}, $$
where $H=\{1,0,-\frac{1}{2},-1,-2\}$, given by
$$ \Phi(z_1,z_2,z_3)=(z_2+z_3-2z_1,z_1+z_3-2z_2,z_1+z_2+z_3),\, \Psi(X,Y,Z)=
     (\dfrac{X}{Y},Y,Z).   $$
\end{lemma}
\begin{proof}
It is clear that $\Phi(\mathcal{A})=\widehat{\mathcal{A}}\times\mathbb{C}$ and 
$H\subset\mathbb{C}$ correspond to $\widehat{\mathcal{A}}$ through $\Psi$; $(4,1,4)$ 
is base point in $S$ and $T$ corresponding to $(0,1,3)\in Q\mathcal{F}_3$. 
\end{proof}
We define five functions $a_h:[0,1]\to\mathbb{C}$, $h\in H$: 
$$ a_h(t)=\begin{cases}
    4-8t                                    & \mbox{ if }t\in[0,\frac{1}{3}],           \\
    \frac{h+1}{2}+\frac{5-3h}{6}e^{6\pi it} & \mbox{ if }t\in[\frac{1}{3},\frac{2}{3}], \\
    8t-4                                    & \mbox{ if }t\in[\frac{2}{3},1].           \\
\end{cases}      $$
As generators of $\pi_1(T)\cong F(5)\times F(1)$ we choose $\hat{\alpha}_h(t)=(a_h(t),1,4)$, 
$h\in H$, and $\hat{\beta}(t)=(4,e^{2\pi it},4)$:
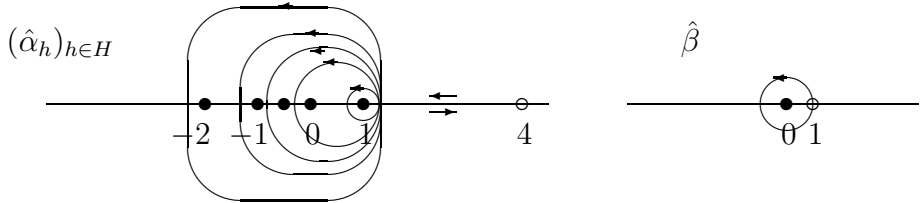
\begin{center}
\begin{picture}(360,110) 
\multiput(77,57)(20,0){4}{$\bullet $}        \put(20,60){\line(1,0){190}} 
\multiput(107,57)(190,0){2}{$\bullet$}       \put(240,60){\line(1,0){110}} 
\multiput(197,57)(110,0){2}{$\circ$}         
\multiput(138,45)(170,0){2}{$1$}             \put(89,45){$-1$} 
\multiput(118,45)(180,0){2}{$0$}             \put(67,45){$-2$} 
\put(198,45){$4$} \put(300,60){\circle{20}}  \put(300,70){\vector(-1,0){5}}
\put(260,80){$\hat{\beta}$}                  \put(5,80){$(\hat{\alpha}_h)_{h\in H}$}
\put(140,60){\oval(12,12)}                   \put(130,60){\oval(32,32)} 
\put(125,60){\oval(43,43)}                   \put(175,63){\vector(-1,0){10}}
\put(165,57){\vector(1,0){10}}               \put(140,66){\vector(-1,0){5}}
\put(130,76){\vector(-1,0){5}}               \put(125,80){\vector(-1,0){5}}
\put(120,60){\oval(53,53)}                   \put(110,60){\oval(73,73)}
\put(123,87){\vector(-1,0){5}}               \put(113,97){\vector(-1,0){5}} 
\end{picture} 
\end{center}
The corresponding pure braids $\alpha_h$ and $\beta$ are the generators of $Q\mathcal{F}_3$;
for an element $\hat{\varepsilon}$ in $\pi_1(T)$, we denote with $\varepsilon$  the pure
braid $\Phi^{-1}_*\Psi^{-1}_*(\hat{\varepsilon})$.

For the second part of the theorem, we use the formulae
$$ \alpha_h(t)=\left(\frac{4-a_h(t)}{3},1,\frac{a_h(t)+5}{3}\right),\,\beta(t)=\left(
          \frac{4-4e^{2\pi it}}{3},\frac{4-e^{2\pi it}}{3},\frac{5e^{2\pi it}+4}{3}\right)  $$
and the following pictures (we have to order the three components of $a_h(\frac{1}{2})$, 
for each $h\in H$):
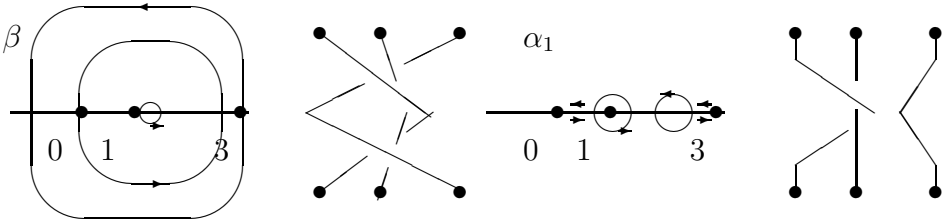
\begin{center}
\begin{picture}(360,80)     
\multiput(20,40)(180,0){2}{\put(7,-3){$\bullet $}\put(27,-3){$\bullet $}\put(67,-3){$\bullet $}
         \put(-17,0){\line(1,0){90}}\put(-3,-18){$0$}\put(17,-18){$1$}\put(60,-18){$3$}} 
\multiput(117,7)(180,0){2}{\put(0,0){$\bullet$}\put(23,0){$\bullet$}\put(53,0){$\bullet$}} 
\multiput(117,67)(180,0){2}{\put(0,0){$\bullet$}\put(23,0){$\bullet$}\put(53,0){$\bullet$}} 
\multiput(231,40)(23,0){2}{\circle{14}}      \put(56,40){\circle{8}}
\put(56,40){\oval(54,54)}                    \put(51,40){\oval(80,80)} 
\put(120,70){\line(4,-3){42}}                \put(115,40){\line(2,-1){60}}
\put(120,10){\line(4,3){18}}                 \put(163,40){\line(-4,-3){10}} 
\put(115,40){\line(2,1){22}}                 \put(175,70){\line(-2,-1){23}} 
\put(143,70){\line(1,-3){6}}                 \put(142,7){\line(1,3){4}}               
\put(153,40){\line(-1,-3){4}}                \put(330,40){\line(-1,-1){9}} 
\put(300,20){\line(3,2){20}}                 \put(300,60){\line(3,-2){30}}
\put(353,60){\line(-2,-3){13.5}}             \put(353,20){\line(-2,3){13.5}}
\put(323,10){\line(0,1){31}}                 \put(323,52){\line(0,1){18}}
\put(0,65){$\beta$}                          \put(197,65){$\alpha_1$} 
\multiput(56,13)(0,22){2}{\vector(1,0){5}}   \put(233,33){\vector(1,0){5}}
\put(56,80){\vector(-1,0){5}}                \put(254,47){\vector(-1,0){5}}
\multiput(215,37)(48,0){2}{\vector(1,0){5}}  \multiput(220,43)(48,0){2}{\vector(-1,0){5}}
\multiput(300,10)(0,50){2}{\line(0,1){10}}   \multiput(353,10)(0,50){2}{\line(0,1){10}} 
\end{picture} 
\end{center}
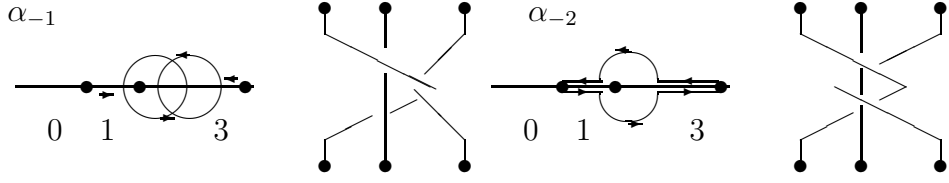
\begin{center}
\begin{picture}(360,80)     
\multiput(20,40)(180,0){2}{\put(7,-3){$\bullet $}\put(27,-3){$\bullet $}\put(67,-3){$\bullet $}
         \put(-17,0){\line(1,0){90}}\put(-5,-20){$0$}\put(15,-20){$1$}\put(58,-20){$3$}} 
\multiput(117,7)(180,0){2}{\put(0,0){$\bullet$}\put(23,0){$\bullet$}\put(53,0){$\bullet$}} 
\multiput(117,67)(180,0){2}{\put(0,0){$\bullet$}\put(23,0){$\bullet$}\put(53,0){$\bullet$}} 
\multiput(56,40)(13,0){2}{\circle{22}}        \put(235,38){\oval(22,22)[b]}
\multiput(210,38)(0,4){2}{\line(1,0){15}}     \put(235,42){\oval(22,22)[t]}
\put(120,60){\line(2,-1){42}}                 \put(173,60){\line(-1,-1){15}}
\put(120,20){\line(2,1){18}}                  \put(145,30){\line(2,1){10}} 
\put(173,20){\line(-1,1){19}}                 \put(340,40){\line(-2,-1){10}} 
\put(143,70){\line(0,-1){15}}                 \put(143,10){\line(0,1){35}}               
\put(300,20){\line(2,1){20}}                  \put(300,60){\line(2,-1){40}}
\put(353,60){\line(-2,-1){23}}                \put(353,20){\line(-2,1){40}}
\put(323,10){\line(0,1){23}}                  \put(323,37){\line(0,1){10}}
\multiput(246,38)(0,4){2}{\line(1,0){23}}     \put(323,70){\line(0,-1){17}}
\multiput(120,10)(0,50){2}{\line(0,1){10}}    \multiput(173,10)(0,50){2}{\line(0,1){10}}
\multiput(300,10)(0,50){2}{\line(0,1){10}}    \multiput(353,10)(0,50){2}{\line(0,1){10}}
\put(0,65){$\alpha_{-1}$}                     \put(197,65){$\alpha_{-2}$} 
\put(215,38){\vector(1,0){5}}                 \put(235,26){\vector(1,0){5}}
\put(220,42){\vector(-1,0){5}}                \put(255,38){\vector(1,0){5}}
\put(260,42){\vector(-1,0){5}}                \put(235,54){\vector(-1,0){5}}
\put(58,28){\vector(1,0){5}}                  \put(69,52){\vector(-1,0){5}}
\put(35,37){\vector(1,0){5}}                  \put(87,43){\vector(-1,0){5}}
\end{picture} 
\end{center}
The pictures of $\alpha_0$ and of $\alpha_{-\frac{1}{2}}$ are similar with those of $\alpha_1$
and $\alpha_{-1}$ respectively. From these pictures we find the relations
$$ \begin{array}{l} 
   p_*\tilde{\textit{\j}}_*(\beta)=x_2x_1x_2^2x_1x_2=\Delta_3^2=x_1^2(x_2x_1^2x_2^{-1})x_2^2=
                                                                      p_*(A_{12}A_{13}A_{23})  \\
   \tilde{\textit{\j}}_*(\alpha_1)=\tilde{\textit{\j}}_*(\alpha_0)=A_{12}                      \\
   p_*\tilde{\textit{\j}}_*(\alpha_{-1})=   p_*\tilde{\textit{\j}}_*(\alpha_{-\frac{1}   
   {2}})=x_1x_2^2x_1=x_1x_2\Delta_3x_2^{-1}=x_1\Delta_3x_1x_2^{-1}=p_*(A_{12}A_{13})           \\
   p_*\tilde{\textit{\j}}_*(\alpha_{-2})=x_1x_2x_1^2x_2x_1=\Delta_3^2=p_*(A_{12}A_{13}A_{23}), 
\end{array}     $$
where $p_*$ is injective, therefore we get the values of $\tilde{\textit{\j}}_*(\beta)$ and
$\tilde{\textit{\j}}_*(\alpha_h)$.                                            

$\bf{Q\mathcal{B}_3.}$ The corresponding base point in $Q\mathcal{C}_3$ (and in $\mathcal{C}_3$)
is $\{0,1,3\}$, that is the polynomial $4X(X^2-4X+3)$.
The group $Q\mathcal{B}_3$ is generated by $\alpha_h$, $\beta$, and also by two new braids
$\gamma_1$, $\gamma_2$ (going to the generators $(12)$, $(23)$ of $\Sigma_3$):
$$ \gamma_1(t)=\left\lbrace\frac{1}{2}-\frac{1}{2}e^{\pi it},\frac{1}{2}+\frac{1}{2}e^{\pi it},3
                                  \right\rbrace,\,\gamma_2(t)=\{0,2-e^{\pi it},2+e^{\pi it}\}, $$
or equivalently 
$$ \gamma_1(t)=4(X-3)\left(X^2-X+\frac{1}{4}-\frac{1}{4}e^{2\pi it}\right),\,
                                       \gamma_2(t)=4X(X^2-4X+4-e^{2\pi it}). $$                                  
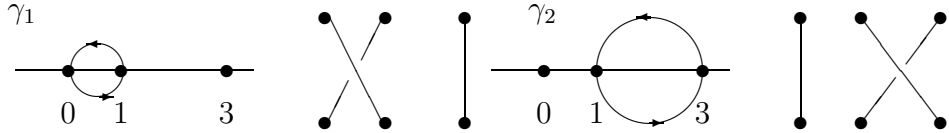
\begin{center}
\begin{picture}(360,60)     
\multiput(20,30)(180,0){2}{\put(0,-3){$\bullet $}\put(20,-3){$\bullet $}\put(60,-3){$\bullet $}
         \put(-17,0){\line(1,0){90}}\put(0,-20){$0$}\put(20,-20){$1$}\put(60,-20){$3$}} 
\multiput(117,7)(180,0){2}{\put(0,0){$\bullet$}\put(23,0){$\bullet$}\put(53,0){$\bullet$}} 
\multiput(117,47)(180,0){2}{\put(0,0){$\bullet$}\put(23,0){$\bullet$}\put(53,0){$\bullet$}} 
\put(243,30){\circle{40}}
\multiput(173,10)(127,0){2}{\line(0,1){40}}    \put(34,30){\circle{20}}
\put(122,50){\line(1,-2){20}}                  \put(323,50){\line(3,-4){30}}
\put(121,10){\line(1,2){8}}                    \put(323,10){\line(3,4){13}} 
\put(142,50){\line(-1,-2){8}}                  \put(353,50){\line(-3,-4){13}} 
\put(0,50){$\gamma_1$}                         \put(197,50){$\gamma_2$} 
\put(243,10){\vector(1,0){5}}                  \put(243,50){\vector(-1,0){5}}
\put(35,20){\vector(1,0){5}}                   \put(35,40){\vector(-1,0){5}}
\end{picture} 
\end{center}
It is clear that $j_*(\gamma_k)=x_k$, $k=1,2$.  Part of the proof of Theorem \ref{thm41}
are given, with full details, in the following lemma.
\begin{lemma}\label{lema41}
In $Q\mathcal{B}_3$ there are the relations
$$ \gamma_2\alpha_1\gamma^{-1}_2=\alpha_0^{-1}\alpha_{-\frac{1}{2}},\,
   \gamma_1\gamma_2\gamma_1\gamma_2^{-1}\gamma_1^{-1}\gamma_2^{-1}=\alpha_{-2}\beta^{-1}. $$
\end{lemma}
\begin{proof}
The following lift of $\gamma_2\alpha_1\gamma^{-1}_2$ in $Q\mathcal{P}_3$:
$$ \gamma_2\alpha_1\gamma^{-1}_2(t)= \begin{cases}
    (0,2-e^{3\pi it},2+e^{3\pi it})         & \mbox{ if }t\in[0,\frac{1}{3}],            \\
    (8t-\frac{8}{3},\frac{17}{3}-8t,1)      & \mbox{ if }t\in[\frac{1}{3},\frac{4}{9}],  \\
    (1-\frac{1}{9}e^{18\pi it},2+\frac{1}{9}e^{18\pi it},1) 
                                            & \mbox{ if }t\in[\frac{4}{9},\frac{5}{9}],  \\
    (\frac{16}{3}-8t,8t-\frac{7}{3},1)      & \mbox{ if }t\in[\frac{5}{9},\frac{2}{3}],  \\                                        
    (0,2+e^{-3\pi it},2-e^{-3\pi it})       & \mbox{ if }t\in[\frac{2}{3},1]             \\
\end{cases}        $$
is transformed through $\Psi_*\Phi_*$ into 
$$ \widehat{\gamma_2\alpha_1\gamma^{-1}_2}(t)= \begin{cases}
    (\frac{4}{-2+3e^{3\pi it}},-2+3e^{3\pi it},4)   & \mbox{ if }t\in[0,\frac{1}{3}],    \\
    (\frac{12-24t}{24t-13},24t-13,4)        & \mbox{ if }t\in[\frac{1}{3},\frac{4}{9}],  \\
    (\frac{e^{18\pi it}+3}{-e^{18\pi it}-6},-2-\frac{1}{3}e^{18\pi it},4) 
                                            & \mbox{ if }t\in[\frac{4}{9},\frac{5}{9}],  \\
    (\frac{24t-12}{11-24t},11-24t,4)        & \mbox{ if }t\in[\frac{5}{9},\frac{2}{3}],  \\                                        
    (\frac{4}{-2-3e^{-3\pi it}},-2-3e^{-3\pi it},4) & \mbox{ if }t\in[\frac{2}{3},1].    \\
\end{cases}        $$
To draw the two components of $\widehat{\gamma_2\alpha_1\gamma^{-1}_2}$ we
have to see when $\widehat{\gamma_2\alpha_1\gamma^{-1}_2}(t)$ is real, and also to 
decide if the complex arcs are either over or under the real line; for this we compute, 
for instance, the values at $t=\frac{1}{6},\frac{17}{36},\frac{19}{36},\frac{5}{6}$.
The picture of the two components of $\widehat{\gamma_2\alpha_1\gamma^{-1}_2}$ 
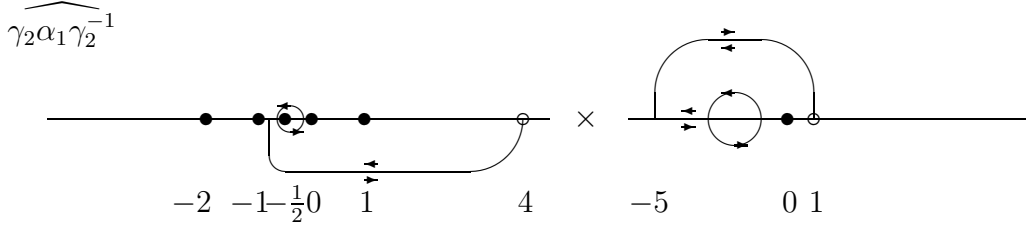
\begin{center}
\begin{picture}(360,80)     
\multiput(77,37)(20,0){4}{$\bullet $}        \put(20,40){\line(1,0){190}} 
\multiput(107,37)(190,0){2}{$\bullet$}       \put(240,40){\line(1,0){150}} 
\multiput(197,37)(110,0){2}{$\circ$}         \put(219,37){$\times$} 
\multiput(138,5)(170,0){2}{$1$}              \put(89,5){$-1$} 
\multiput(118,5)(180,0){2}{$0$}              \put(67,5){$-2$} 
\put(198,5){$4$}                             \put(102,5){$-\frac{1}{2}$}
\put(280,40){\oval(60,60)[t]}                \put(110,40){\oval(180,40)[br]}
\put(110,40){\oval(12,40)[bl]}               \put(280,67){\vector(-1,0){5}}
\put(5,70){$\widehat{\gamma_2\alpha_1\gamma^{-1}_2}$}
\put(112,40){\circle{10}}                    \put(280,40){\circle{20}}
\put(112,35){\vector(1,0){5}}                \put(280,30){\vector(1,0){5}} 
\put(112,45){\vector(-1,0){5}}               \put(280,50){\vector(-1,0){5}}
\put(260,37){\vector(1,0){5}}                \put(265,43){\vector(-1,0){5}}
\put(145,23){\vector(-1,0){5}}               \put(140,17){\vector(1,0){5}}
\put(275,73){\vector(1,0){5}}                \put(240,5){$-5$} 
\end{picture} 
\end{center}
shows that $\widehat{\gamma_2\alpha_1\gamma^{-1}_2}
\widehat{\alpha_{-\frac{1}{2}}^{-1}}=\widehat{\alpha_0^{-1}}$, 
hence $\gamma_2\alpha_1\gamma^{-1}_2=\alpha_0^{-1}\alpha_{-\frac{1}{2}}$.

For the second relation we take the lift
$$ \gamma_1\gamma_2\gamma_1\gamma_2^{-1}\gamma_1^{-1}\gamma_2^{-1}(t)= \begin{cases}
    (\frac{1}{2}-\frac{1}{2}e^{6\pi it},\frac{1}{2}+\frac{1}{2}e^{6\pi it},3)
                                            & \mbox{ if }t\in[0,\frac{1}{6}],            \\
    (2+e^{6\pi it},0,2-e^{6\pi it})         & \mbox{ if }t\in[\frac{1}{6},\frac{1}{3}],  \\
    (3,\frac{1}{2}-\frac{1}{2}e^{6\pi it},\frac{1}{2}+\frac{1}{2}e^{6\pi it}) 
                                            & \mbox{ if }t\in[\frac{1}{3},\frac{1}{2}],  \\
    (2-e^{-6\pi it},2+e^{-6\pi it},0)       & \mbox{ if }t\in[\frac{1}{2},\frac{2}{3}],  \\                                        
    (\frac{1}{2}+\frac{1}{2}e^{-6\pi it},3,\frac{1}{2}-\frac{1}{2}e^{-6\pi it})
                                            & \mbox{ if }t\in[\frac{2}{3},\frac{5}{6}], \\
    (0,2-e^{-6\pi it},2+e^{-6\pi it})       & \mbox{ if }t\in[\frac{5}{6},1]             \\
\end{cases}        $$
and its image in $\pi_1(T)$
$$ \widehat{\gamma_1\gamma_2\gamma_1\gamma_2^{-1}\gamma_1^{-1}\gamma_2^{-1}}(t)= 
 \begin{cases}
    (\frac{5+3e^{6\pi it}}{5-3e^{6\pi it}},\frac{5}{2}-\frac{3}{2}e^{6\pi it},4)
                                      & \mbox{ if }t\in[0,\frac{1}{6}],            \\
    (-\frac{1}{2}-\frac{3}{4}e^{6\pi it},4,4)  
                                      & \mbox{ if }t\in[\frac{1}{6},\frac{1}{3}],  \\
    (\frac{-10}{5+3e^{6\pi it}},\frac{5}{2}+\frac{3}{2}e^{6\pi it},4) 
                                      & \mbox{ if }t\in[\frac{1}{3},\frac{1}{2}],  \\
    (\frac{2-3e^{-6\pi it}}{2+3e^{-6\pi it}},-2-3e^{-6\pi it},4)
                                      & \mbox{ if }t\in[\frac{1}{2},\frac{2}{3}],  \\                                        
    (-\frac{1}{2}+\frac{3}{10}e^{-6\pi it},-5,4)
                                      & \mbox{ if }t\in[\frac{2}{3},\frac{5}{6}],  \\
    (\frac{4}{-2+3e^{-6\pi it}},-2+3e^{-6\pi it},4)
                                      & \mbox{ if }t\in[\frac{5}{6},1].             \\
\end{cases}        $$
The values at $t=\frac{1}{12},\frac{1}{4},\frac{5}{12},\frac{7}{12},\frac{5}{6},\frac{11}{12}$
give the under- and over-arcs in the picture of the components of 
$\Phi_*\Psi_*(\gamma_1\gamma_2\gamma_1\gamma_2^{-1}\gamma_1^{-1}\gamma_2^{-1})$ 
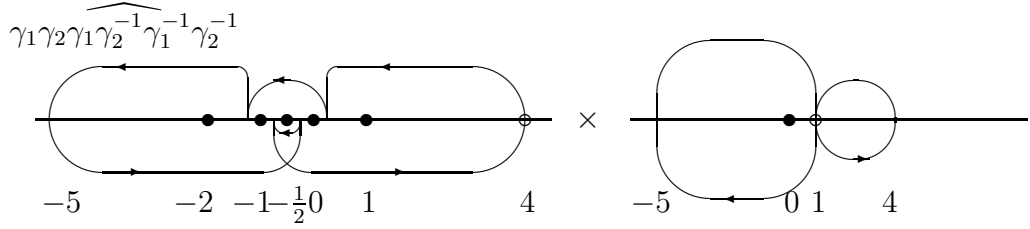
\begin{center}
\begin{picture}(360,80)      
\multiput(77,37)(20,0){4}{$\bullet $}        \put(15,40){\line(1,0){195}} 
\multiput(107,37)(190,0){2}{$\bullet$}       \put(240,40){\line(1,0){150}} 
\multiput(197,37)(110,0){2}{$\circ$}         \put(219,37){$\times$} 
\multiput(138,5)(170,0){2}{$1$}              \put(89,5){$-1$} 
\multiput(118,5)(180,0){2}{$0$}              \put(67,5){$-2$} 
\put(198,5){$4$}                             \put(102,5){$-\frac{1}{2}$}
\put(280,40){\oval(60,60)}                   \put(325,40){\oval(30,30)}
\put(130,40){\oval(140,40)[tr]}              \put(130,40){\oval(10,40)[tl]}
\put(110,40){\oval(30,30)[t]}                \put(110,40){\oval(10,10)[b]}
\put(90,40){\oval(10,40)[tr]}                \put(90,40){\oval(140,40)[tl]} 
\put(100,40){\oval(160,40)[bl]}              \put(100,40){\oval(30,40)[br]}
\put(120,40){\oval(30,40)[bl]}               \put(120,40){\oval(160,40)[br]} 
\put(5,70){$\widehat{\gamma_1\gamma_2\gamma_1\gamma_2^{-1}\gamma_1^{-1}\gamma_2^{-1}}$}
\put(280,10){\vector(-1,0){5}}               \put(325,25){\vector(1,0){5}}
\put(110,55){\vector(-1,0){5}}               \put(112,35){\vector(-1,0){5}} 
\put(150,60){\vector(-1,0){5}}               \put(150,20){\vector(1,0){5}}
\put(50,60){\vector(-1,0){5}}                \put(50,20){\vector(1,0){5}}
\put(240,5){$-5$} 
\put(17,5){$-5$}                              \put(335,5){$4$} 
\end{picture} 
\end{center}
From this we find that $\widehat{\gamma_1\gamma_2\gamma_1\gamma_2^{-1}\gamma_1^{-1}
\gamma_2^{-1}}=\hat{\alpha}_{-2}\hat{\beta}^{-1}$, therefore we have the relation
$\gamma_1\gamma_2\gamma_1\gamma_2^{-1}\gamma_1^{-1}\gamma_2^{-1}=\alpha_{-2}\beta^{-1}$.
\end{proof}
For the remaining relations in $Q\mathcal{B}_3$ we will use the same method: for
the other $\hat{\varepsilon}$'s we will list their pictures, we will write the corresponding 
$\varepsilon$'s relations, but we will omit all the long formulae. 

From the exact sequence
$$ 1\rightarrow Q\mathcal{P}_3\hookrightarrow Q\mathcal{B}_3\rightarrow\Sigma_3\rightarrow 1, $$
where $\Sigma_3$ is presented as
$$ \Sigma_3=\langle (12),(23)\mid (12)^2=1,(23)^2=1, (12)(23)(12)=(23)(12)(23)\rangle, $$
there is a presentation of $Q\mathcal{B}_3$ of the form
$$ Q\mathcal{B}_3=\left\langle \begin{array}{ccc}   
    \alpha_1,\alpha_0,\alpha_{-\frac{1}{2}}, & \mid & [\alpha_h,\beta]=1,      \\
       \alpha_{-1},\alpha_{-2},\beta,        & \mid &                               
       \gamma^{\pm 1}_k\alpha_h\gamma^{\mp 1}_k\in Q\mathcal{P}_3,
       \gamma^{\pm 1}_k\beta\gamma^{\mp 1}_k\in Q\mathcal{P}_3,                \\
    \gamma_1,\gamma_2                        & \mid &
       \gamma_k^2\in Q\mathcal{P}_3,\gamma_1\gamma_2\gamma_1\gamma_2^{-1}\gamma_1^{-1}
       \gamma_2^{-1}\in Q\mathcal{P}_3
\end{array}  \right\rangle,  $$ 
where $h\in\{1,0,-\frac{1}{2},-1,-2\}$ and $k\in\{1,2\}$.

From the picture
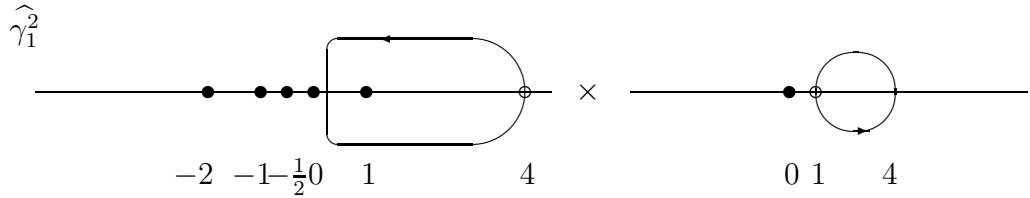
\begin{center}
\begin{picture}(360,80)   
\multiput(77,37)(20,0){4}{$\bullet $}        \put(15,40){\line(1,0){195}} 
\multiput(107,37)(190,0){2}{$\bullet$}       \put(240,40){\line(1,0){150}} 
\multiput(197,37)(110,0){2}{$\circ$}         \put(219,37){$\times$} 
\multiput(138,5)(170,0){2}{$1$}              \put(89,5){$-1$} 
\multiput(118,5)(180,0){2}{$0$}              \put(67,5){$-2$} 
\put(198,5){$4$}                             \put(102,5){$-\frac{1}{2}$}
\put(325,40){\oval(30,30)}                   \put(335,5){$4$} 
\put(130,40){\oval(140,40)[r]}               \put(130,40){\oval(10,40)[l]}
\put(5,60){$\widehat{\gamma_1^2}$}
\put(150,60){\vector(-1,0){5}}               \put(325,25){\vector(1,0){5}}
\end{picture} 
\end{center}
we find that $\gamma_1^2=\alpha_1$ and this implies 
$\gamma^{\pm 1}_1\alpha_1\gamma^{\mp 1}_1=\alpha_1$.

The relation $\gamma_2^2=\alpha_{-\frac{1}{2}}^{-1}\beta$ is given by the picture
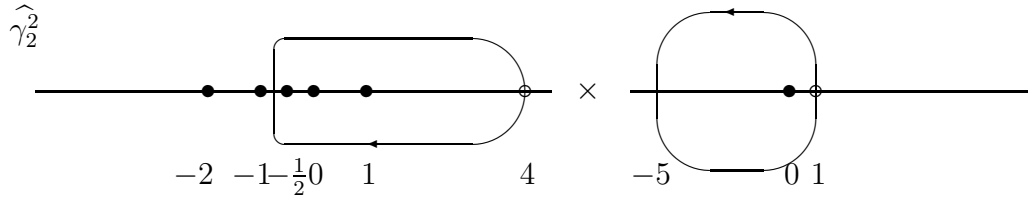
\begin{center}
\begin{picture}(360,80)   
\multiput(77,37)(20,0){4}{$\bullet $}        \put(15,40){\line(1,0){195}} 
\multiput(107,37)(190,0){2}{$\bullet$}       \put(240,40){\line(1,0){150}} 
\multiput(197,37)(110,0){2}{$\circ$}         \put(219,37){$\times$} 
\multiput(138,5)(170,0){2}{$1$}              \put(89,5){$-1$} 
\multiput(118,5)(180,0){2}{$0$}              \put(67,5){$-2$} 
\put(198,5){$4$}                             \put(102,5){$-\frac{1}{2}$}
\put(280,40){\oval(60,60)}                   \put(240,5){$-5$} 
\put(110,40){\oval(180,40)[r]}               \put(110,40){\oval(10,40)[l]}
\put(5,60){$\widehat{\gamma_2^2}$}
\put(280,70){\vector(-1,0){5}}               \put(145,20){\vector(-1,0){5}} 
\end{picture} 
\end{center}

The following two pictures show that $[\gamma_1,\beta]=1$ and 
$[\gamma_2,\beta]=1$, hence $\beta$ is central in $Q\mathcal{B}_3$:
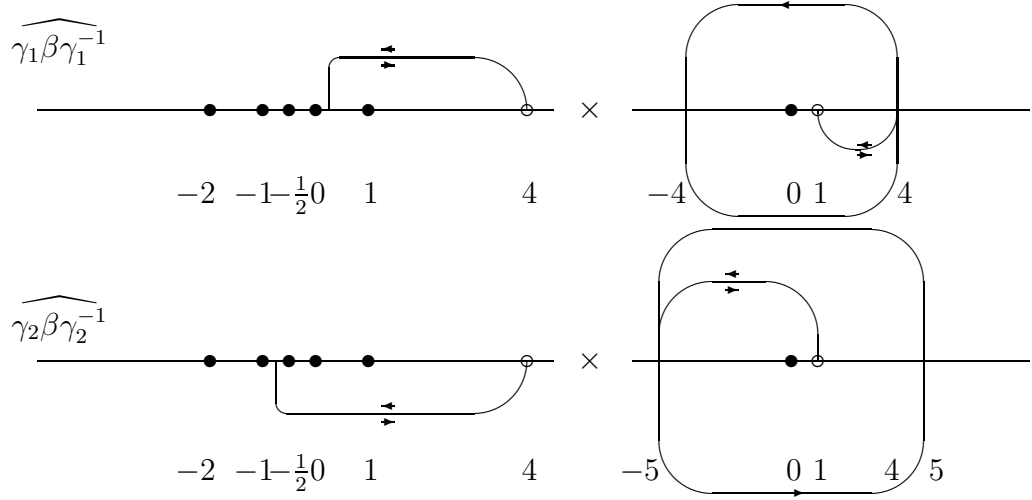
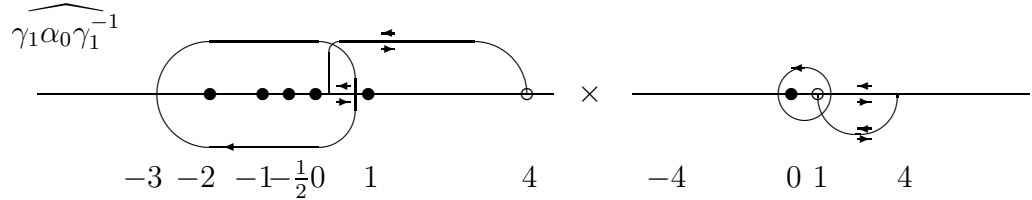
\begin{center}
\begin{picture}(360,80)    
\multiput(77,37)(20,0){4}{$\bullet $}        \put(15,40){\line(1,0){195}} 
\multiput(107,37)(190,0){2}{$\bullet$}       \put(240,40){\line(1,0){150}} 
\multiput(197,37)(110,0){2}{$\circ$}         \put(219,37){$\times$} 
\multiput(138,5)(170,0){2}{$1$}              \put(89,5){$-1$} 
\multiput(118,5)(180,0){2}{$0$}              \put(67,5){$-2$} 
\put(198,5){$4$}                             \put(102,5){$-\frac{1}{2}$}
\put(300,40){\oval(80,80)}                   \put(325,40){\oval(30,30)[b]}
\put(340,5){$4$}                             \put(300,80){\vector(-1,0){5}}
\put(130,40){\oval(140,40)[tr]}              \put(130,40){\oval(10,40)[tl]}
\put(5,60){$\widehat{\gamma_1\beta\gamma_1^{-1}}$}
\put(150,63){\vector(-1,0){5}}               \put(145,57){\vector(1,0){5}} 
\put(325,23){\vector(1,0){5}}                \put(330,27){\vector(-1,0){5}}
\put(245,5){$-4$}                             
\end{picture} 
\end{center}
\begin{center}
\begin{picture}(360,100)    
\multiput(77,47)(20,0){4}{$\bullet $}        \put(15,50){\line(1,0){195}} 
\multiput(107,47)(190,0){2}{$\bullet$}       \put(240,50){\line(1,0){150}} 
\multiput(197,47)(110,0){2}{$\circ$}         \put(219,47){$\times$} 
\multiput(138,5)(170,0){2}{$1$}              \put(89,5){$-1$} 
\multiput(118,5)(180,0){2}{$0$}              \put(67,5){$-2$} 
\put(198,5){$4$}                             \put(102,5){$-\frac{1}{2}$}
\put(300,50){\oval(100,100)}                 \put(280,50){\oval(60,60)[t]}   
\put(335,5){$4$}                             \put(300,0){\vector(1,0){5}}
\put(110,50){\oval(180,40)[br]}              \put(110,50){\oval(10,40)[bl]}
\put(5,60){$\widehat{\gamma_2\beta\gamma_2^{-1}}$}
\put(150,33){\vector(-1,0){5}}               \put(145,27){\vector(1,0){5}} 
\put(280,83){\vector(-1,0){5}}               \put(275,77){\vector(1,0){5}}
\put(235,5){$-5$}                            \put(352,5){$5$} 
\end{picture} 
\end{center}
From the picture
\begin{center}
\begin{picture}(360,80)    
\multiput(77,37)(20,0){4}{$\bullet $}        \put(15,40){\line(1,0){195}} 
\multiput(107,37)(190,0){2}{$\bullet$}       \put(240,40){\line(1,0){150}} 
\multiput(197,37)(110,0){2}{$\circ$}         \put(219,37){$\times$} 
\multiput(138,5)(170,0){2}{$1$}              \put(89,5){$-1$} 
\multiput(118,5)(180,0){2}{$0$}              \put(67,5){$-2$} 
\put(198,5){$4$}                             \put(102,5){$-\frac{1}{2}$}
\put(130,40){\oval(140,40)[tr]}              \put(130,40){\oval(10,40)[tl]}
\put(120,40){\oval(120,40)[l]}               \put(120,40){\oval(30,40)[r]}
\put(325,40){\oval(30,30)[b]}                \put(305,40){\oval(20,20)}
\put(340,5){$4$}                             \put(47,5){$-3$} 
\put(90,20){\vector(-1,0){5}}                \put(305,50){\vector(-1,0){5}}
\put(128,37){\vector(1,0){5}}                \put(133,43){\vector(-1,0){5}}
\put(5,60){$\widehat{\gamma_1\alpha_0\gamma_1^{-1}}$}
\put(150,63){\vector(-1,0){5}}               \put(145,57){\vector(1,0){5}} 
\put(325,23){\vector(1,0){5}}                \put(330,27){\vector(-1,0){5}}
\put(325,37){\vector(1,0){5}}                \put(330,43){\vector(-1,0){5}}
\put(245,5){$-4$}         
\end{picture} 
\end{center}
we get $\widehat{\alpha_1^{-1}}\widehat{\gamma_1\alpha_0\gamma_1^{-1}}=
\widehat{\alpha_{-2}^{-1}}\widehat{\beta}$, hence 
$\gamma_1\alpha_0\gamma_1^{-1}=\alpha_1\alpha_{-2}^{-1}\beta$, and from
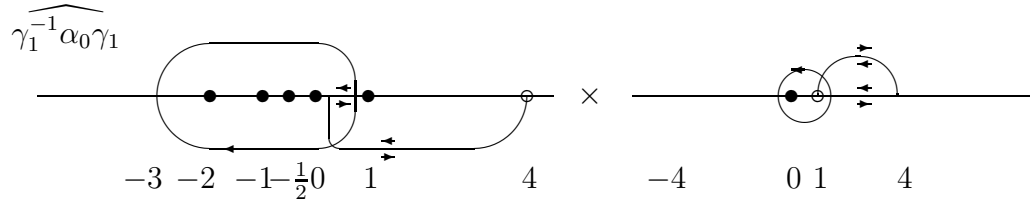
\begin{center}
\begin{picture}(360,80)            
\multiput(77,37)(20,0){4}{$\bullet $}        \put(15,40){\line(1,0){195}} 
\multiput(107,37)(190,0){2}{$\bullet$}       \put(240,40){\line(1,0){150}} 
\multiput(197,37)(110,0){2}{$\circ$}         \put(219,37){$\times$} 
\multiput(138,5)(170,0){2}{$1$}              \put(89,5){$-1$} 
\multiput(118,5)(180,0){2}{$0$}              \put(67,5){$-2$} 
\put(198,5){$4$}                             \put(102,5){$-\frac{1}{2}$}
\put(130,40){\oval(140,40)[br]}              \put(130,40){\oval(10,40)[bl]}
\put(120,40){\oval(120,40)[l]}               \put(120,40){\oval(30,40)[r]}
\put(325,40){\oval(30,30)[t]}                \put(305,40){\oval(20,20)}
\put(340,5){$4$}                             \put(47,5){$-3$} 
\put(90,20){\vector(-1,0){5}}                \put(305,50){\vector(-1,0){5}}
\put(128,37){\vector(1,0){5}}                \put(133,43){\vector(-1,0){5}}
\put(5,60){$\widehat{\gamma_1^{-1}\alpha_0\gamma_1}$}
\put(150,23){\vector(-1,0){5}}               \put(145,17){\vector(1,0){5}} 
\put(325,58){\vector(1,0){5}}                \put(330,52){\vector(-1,0){5}}
\put(325,37){\vector(1,0){5}}                \put(330,43){\vector(-1,0){5}}
\put(245,5){$-4$}         
\end{picture} 
\end{center}
we get $\widehat{\gamma_1^{-1}\alpha_0\gamma_1}\widehat{\alpha_1^{-1}}=
\widehat{\alpha_{-2}^{-1}}\widehat{\beta}$, hence 
$\gamma_1^{-1}\alpha_0\gamma_1=\alpha_{-2}^{-1}\alpha_1\beta$.

The following picture
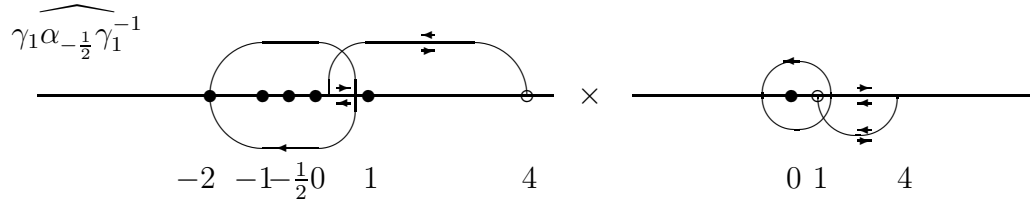
\begin{center}
\begin{picture}(360,80)    
\multiput(77,37)(20,0){4}{$\bullet $}        \put(15,40){\line(1,0){195}} 
\multiput(107,37)(190,0){2}{$\bullet$}       \put(240,40){\line(1,0){150}} 
\multiput(197,37)(110,0){2}{$\circ$}         \put(219,37){$\times$} 
\multiput(138,5)(170,0){2}{$1$}              \put(89,5){$-1$} 
\multiput(118,5)(180,0){2}{$0$}              \put(67,5){$-2$} 
\put(198,5){$4$}                             \put(102,5){$-\frac{1}{2}$}
\put(302,40){\oval(26,26)}                   \put(325,40){\oval(30,30)[b]}
\put(340,5){$4$}                             \put(302,53){\vector(-1,0){5}}
\put(140,40){\oval(120,40)[tr]}              \put(140,40){\oval(30,40)[tl]}
\put(120,40){\oval(30,40)[r]}                \put(120,40){\oval(80,40)[l]}
\put(5,60){$\widehat{\gamma_1\alpha_{-\frac{1}{2}}\gamma_1^{-1}}$}
\put(165,63){\vector(-1,0){5}}               \put(160,57){\vector(1,0){5}} 
\put(325,23){\vector(1,0){5}}                \put(330,27){\vector(-1,0){5}}  
\put(325,43){\vector(1,0){5}}                \put(330,37){\vector(-1,0){5}} 
\put(128,43){\vector(1,0){5}}                \put(133,37){\vector(-1,0){5}}  
\put(110,20){\vector(-1,0){5}}
\end{picture} 
\end{center}
\noindent shows that $\widehat{\alpha_1^{-1}}\widehat{\gamma_1\alpha_{-\frac{1}{2}}
\gamma_1^{-1}}=\widehat{\alpha_{-1}}\widehat{\beta}$, therefore 
$\gamma_1\alpha_{-\frac{1}{2}}\gamma_1^{-1}=\alpha_1\alpha_{-1}^{-1}\beta$.

We have sufficiently many relations to compute $\{\alpha_h\}$ as words in 
$\gamma_1,\gamma_2$ and $\beta$:
$$ \begin{array}{lll}
  \alpha_1=\gamma_1^2                                                                    &                                                            
  \alpha_{-2}=\gamma_1\gamma_2\gamma_1\gamma_2^{-1}\gamma_1^{-1}\gamma_2^{-1}\beta       &
   \alpha_{-\frac{1}{2}}=\gamma_2^{-2}\beta                                           \\
  \alpha_{-1}=\beta\gamma_1\alpha_{-\frac{1}{2}}^{-1}\gamma_1=\gamma_1\gamma_2^2\gamma_1 &
   \alpha_0=\alpha_{-\frac{1}{2}}\gamma_2\alpha_1^{-1}\gamma_2^{-1}=
                                   \gamma_2^{-1}\gamma_1^{-2}\gamma_2^{-1}\beta.         & 
\end{array}  $$
Using $\gamma_1\alpha_0\gamma_1^{-1}=\alpha_1\alpha_{-2}^{-1}\beta$ and 
$\gamma_1^{-1}\alpha_0\gamma_1=\alpha_{-2}^{-1}\alpha_1\beta$ we find the relations
$$ \beta=\gamma_1\gamma_2\gamma_1\gamma_2\gamma_1\gamma_2\mbox{ and }
   \beta=\gamma_2\gamma_1\gamma_2\gamma_1\gamma_2\gamma_1,   $$
therefore $Q\mathcal{B}_3$ is generated by $\gamma_1,\gamma_2$: 
$$ (*)\quad\quad \begin{array}{lll}
   \alpha_1=\gamma_1^2                                                   &           
   \alpha_0=\gamma_1\gamma_2\gamma_1\gamma_2\gamma_1^{-1}\gamma_2^{-1}   &           
   \alpha_{-\frac{1}{2}}=\gamma_1\gamma_2\gamma_1\gamma_2\gamma_1\gamma_2^{-1}      \\
   \alpha_{-1}=\gamma_1\gamma_2^2\gamma_1                                &           
   \alpha_{-2}=\gamma_1\gamma_2\gamma_1^2\gamma_2\gamma_1                &           
   \beta=\gamma_1\gamma_2\gamma_1\gamma_2\gamma_1\gamma_2.
\end{array}  $$
The relation $(\gamma_1\gamma_2)^3=(\gamma_2\gamma_1)^3$ is the unique relation because
the rest of defining relations in $Q\mathcal{B}_3$:
$$ \begin{array}{lll} 
\gamma_1^{-1}\alpha_{-\frac{1}{2}}\gamma_1=\alpha_{-1}^{-1}\alpha_1\beta    & \quad & 
        \gamma_2^{-1}\alpha_1\gamma_2=\alpha_{-\frac{1}{2}}\alpha_0^{-1}             \\
\gamma_2\alpha_0\gamma_2^{-1}=\alpha_1^{-1}\alpha_{-\frac{1}{2}}            & \quad &   
        \gamma_2^{-1}\alpha_0\gamma_2=\alpha_{-\frac{1}{2}}\alpha_1^{-1}             \\
\gamma_2\alpha_{-\frac{1}{2}}\gamma_2^{-1}=\alpha_{-\frac{1}{2}}            & \quad &   
       \gamma_2^{-1}\alpha_{-\frac{1}{2}}\gamma_2=\alpha_{-\frac{1}{2}}              \\        
\gamma_1\alpha_{-1}\gamma_1^{-1}=\alpha_1\alpha_{-\frac{1}{2}}^{-1}\beta    & \quad &   
       \gamma_1^{-1}\alpha_{-1}\gamma_1=\alpha_{-\frac{1}{2}}^{-1}\alpha_1\beta      \\          
\gamma_2\alpha_{-1}\gamma_2^{-1}=\alpha_{-2}^{-1}\alpha_{-\frac{1}{2}}\beta & \quad &   
       \gamma_2^{-1}\alpha_{-1}\gamma_2=\alpha_{-\frac{1}{2}}\alpha_{-2}^{-1}\beta   \\
\gamma_1\alpha_{-2}\gamma_1^{-1}=\alpha_1\alpha_0^{-1}\beta                 & \quad &   
       \gamma_1^{-1}\alpha_{-2}\gamma_1=\alpha_0^{-1}\alpha_1\beta                   \\
\gamma_2\alpha_{-2}\gamma_2^{-1}=\alpha_{-1}^{-1}\alpha_{-\frac{1}{2}}\beta & \quad &   
       \gamma_2^{-1}\alpha_{-2}\gamma_2=\alpha_{-\frac{1}{2}}\alpha_{-1}^{-1}\beta           
\end{array}  $$
can be checked using $(*)$ relations. 

In order to show that the initial long presentation (with eight generators) and the short
presentation (with two generators) are equivalent, we have to verify few relations; for 
instance, we check that $\beta$ is a central element, and for this it is enough to show that
$[\beta,\gamma_k]=1$:
$$ \begin{array}{l}  
  \gamma_1\beta=\gamma_1(\gamma_2\gamma_1\gamma_2\gamma_1\gamma_2\gamma_1)=
    (\gamma_1\gamma_2\gamma_1\gamma_2\gamma_1\gamma_2)\gamma_1=\beta\gamma_1   \\
  \gamma_2\beta=\gamma_2(\gamma_1\gamma_2\gamma_1\gamma_2\gamma_1\gamma_2)=
  (\gamma_2\gamma_1\gamma_2\gamma_1\gamma_2\gamma_1)\gamma_2=\beta\gamma_2. 
\end{array}     $$ 

It is clear that $j_*(\gamma_k)=x_k$. The $(*)$ relations give the values of $Qp_*$. 
\begin{corollary}\label{cor49}
For any three points in the plane and for any deformation 
$H^s(t)=\{h_1^s(t),h_2^s(t),h_3^s(t)\}$ between the braids $H^0=x_1x_2x_1$ and  
$H^1=x_2x_1x_2$, there is a pair $(s,t)$ when one point $h_k^s(t)$ is the midpoint 
of the corresponding points $h_i^s(t),h_j^s(t)$ on the other two threads. 
\end{corollary}
\begin{proof}
If there is deformation where the three points $\{h_1^s(t),h_2^s(t),h_3^s(t)\}$ make never 
an arithmetic progression, then in $Q\mathcal{B}_3$ we have the relation 
$\gamma_1\gamma_2\gamma_1=\gamma_2\gamma_1\gamma_2$: this implies that the abelianization 
of $Q\mathcal{B}_3$ is a cyclic group, but it is clear that 
$H_1(Q\mathcal{C}_3)=\mathbb{Z}\oplus\mathbb{Z}$. 
\end{proof} 
\noindent\textit{Proof of Corollary \ref{cor40}.} a) As a Garside group, $Q\mathcal{B}_3$ is torsion 
free (see \cite{Pa}). If $\eta$ is a central element in $Q\mathcal{B}_3$, then in the sequence
$$ 1\rightarrow Q\mathcal{P}_3\hookrightarrow Q\mathcal{B}_3
                                   \overset{\partial}{\rightarrow}\Sigma_3\rightarrow 1 $$
$\partial(\eta)$ is central in $\Sigma_3$, hence $\eta$ is central in ${\rm ker}(\partial)$, 
and from Theorem \ref{thm41}, the center of $Q\mathcal{P}_3$ is generated by 
$\beta=(\gamma_1\gamma_2)^3$, the Garside element $\bf{\Delta}$ of the monoid 
$Q\mathcal{B}_3^+$. Let us remark the relation between the two Garside elements:
 $j_*({\bf \Delta})=\Delta_3^2$. \hfill$\square$ 

We can now give a proof of a statement from the Introduction.
\begin{corollary}\label{cor50}
The spaces $R\mathcal{F}_3$, $R\mathcal{C}_3$, $Q\mathcal{F}_3$, $Q\mathcal{C}_3$,
$R\mathcal{F}_4$ and $R\mathcal{C}_4$ are $K(\pi,1)$ spaces.
\end{corollary}
\begin{proof} $Q\mathcal{F}_3$ is a $K(F(5)\times F(1),1)$ space from Lemma \ref{lema40}. Using
the coverings
$$ \Sigma_3\hookrightarrow Q\mathcal{F}_3\twoheadrightarrow Q\mathcal{C}_3,\,  
   \Sigma_4\hookrightarrow Q\mathcal{F}_4\twoheadrightarrow Q\mathcal{C}_4 $$
and the fibration
$$ \mathbb{C}\setminus 3\hookrightarrow R\mathcal{C}_4\twoheadrightarrow Q\mathcal{C}_3 $$
we find that $Q\mathcal{C}_3$, $R\mathcal{C}_4$ and $R\mathcal{F}_4$ are $K(\pi,1)$ spaces. The 
proof for $R\mathcal{F}_3$ and $R\mathcal{C}_3$ was given at the beginning of Section \ref{sec3}.
\end{proof} 
 
$\bf{R\mathcal{B}_4.}$ We choose $\{\frac{2-\sqrt{7}}{3},2-\sqrt{3},\frac{2+\sqrt{7}}{3},2+\sqrt{3}\}$ 
(or $P_4(X)=X^4-\frac{16}{3}X^3+6X^2-\frac{1}{3}$) as base point in $R\mathcal{C}_4$ (its derivative, 
$Q_3(X)=4(X^3-4X^2+3X)$, is the base point in $Q\mathcal{C}_3$). 
The fiber over this point is $\mathbb{C}\setminus\{-\frac{5}{3},0,9\}$ and we choose
$-\frac{1}{3}$ as base point in this fiber and the following paths as generators of $F_3$, 
its fundamental group:
$$ \begin{array}{c}
        d_1(t)=\begin{cases} 
(1-3\varepsilon)t-\frac{1}{3}\quad\,\,\quad    & \mbox{ if }t\in[0,\frac{1}{3}],            \\
-\varepsilon e^{6\pi it}     \quad\,\,\quad    & \mbox{ if }t\in[\frac{1}{3},\frac{2}{3}],  \\
(3\varepsilon-1)t+(\frac{2}{3}-3\varepsilon)\quad\,\,\quad   
                                               & \mbox{ if }t\in[\frac{2}{3},1],            \\
              \end{cases}              \\ 
       d_2(t)=\begin{cases} 
(3\varepsilon-4)t-\frac{1}{3} \quad  \quad     & \mbox{ if }t\in[0,\frac{1}{3}],            \\
-\frac{5}{3}+\varepsilon e^{6\pi it}\quad\quad & \mbox{ if }t\in[\frac{1}{3},\frac{2}{3}],  \\
(4-3\varepsilon)t+(3\varepsilon-\frac{13}{3})\quad\quad 
                                               & \mbox{ if }t\in[\frac{2}{3},1],            \\
              \end{cases}              \\
      d_3(t)=\begin{cases} 
(4-12\varepsilon)t-\frac{1}{3}                 & \mbox{ if }t\in[0,\frac{1}{12}],           \\
\varepsilon ie^{6\pi it}                       & \mbox{ if }t\in[\frac{1}{12},\frac{1}{4}], \\
(108-24\varepsilon)t+(7\varepsilon-27)         & \mbox{ if }t\in[\frac{1}{4},\frac{1}{3}],  \\
9-\varepsilon e^{6\pi it}                      & \mbox{ if }t\in[\frac{1}{3},\frac{2}{3}],  \\
(24\varepsilon-108)t+(81-17\varepsilon)        & \mbox{ if }t\in[\frac{2}{3},\frac{3}{4}],  \\
\varepsilon ie^{-6\pi it}                      & \mbox{ if }t\in[\frac{3}{4},\frac{11}{12}],\\
(12\varepsilon-4)t+(\frac{11}{3}-12\varepsilon)& \mbox{ if }t\in[\frac{11}{12},1]           \\
              \end{cases}                
\end{array}     $$  
(one can take $0<\varepsilon<\frac{1}{1000}$). 
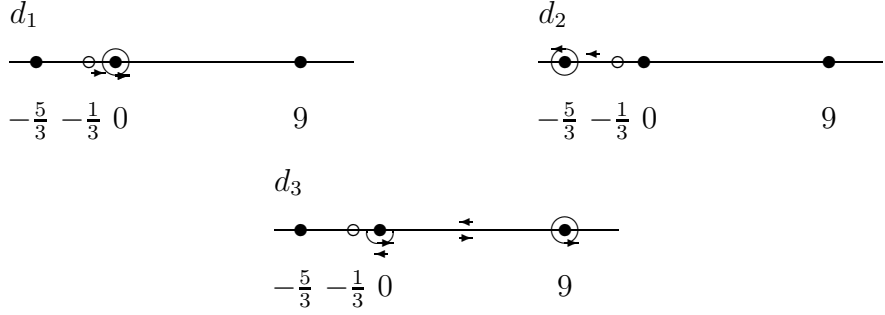
\begin{center}
\begin{picture}(360,60)        
\put(10,30){\line(1,0){130}}    \put(17,27){$\bullet $}       \put(47,27){$\bullet $} 
\put(117,27){$\bullet $}        \put(9,5){$-\frac{5}{3}$}     \put(29,5){$-\frac{1}{3}$} 
\put(49,5){$0$}                 \put(117,5){$9$}              \put(37,27){$\circ$}      
\put(10,45){$d_1$}              \put(41,26){\vector(1,0){5}} 
\put(50,30){\circle{10}}        \put(50,25){\vector(1,0){5}}  \put(50,25){\vector(1,0){5}}
\put(200,0){\put(10,45){$d_2$} 
\put(10,30){\line(1,0){130}}    \put(17,27){$\bullet $}       \put(47,27){$\bullet $} 
\put(117,27){$\bullet $}        \put(9,5){$-\frac{5}{3}$}     \put(29,5){$-\frac{1}{3}$} 
\put(49,5){$0$}                 \put(117,5){$9$}              \put(37,27){$\circ$}      
\put(20,30){\circle{10}}        \put(20,35){\vector(-1,0){5}} \put(33,33){\vector(-1,0){5}}} 
\end{picture}
\end{center}    
\begin{center}
\begin{picture}(360,60)        
\put(100,0){\put(49,5){$0$}
\put(10,30){\line(1,0){130}}    \put(17,27){$\bullet $}       \put(47,27){$\bullet $} 
\put(117,27){$\bullet $}        \put(9,5){$-\frac{5}{3}$}     \put(29,5){$-\frac{1}{3}$} 
\put(120,25){\vector(1,0){5}}   \put(117,5){$9$}              \put(37,27){$\circ$}      
\put(10,45){$d_3$}              \put(50,30){\oval(10,10)[b]}  \put(120,30){\circle{10}}        
\put(50,25){\vector(1,0){5}}    \put(53,21){\vector(-1,0){5}} \put(80,27){\vector(1,0){5}}  \put(85,33){\vector(-1,0){5}}}
\end{picture}
\end{center}

Let us introduce the braids 
$\Gamma_1,\Gamma_2$: 
$$ \begin{array}{l}
\Gamma_1(t)=X^4-\frac{16}{3}X^3+(\frac{13}{2}-\frac{1}{2}e^{2\pi it})X^2+
                                                      (3e^{2\pi it}-3)X-\frac{1}{3},  \\
\Gamma_2(t)=X^4-\frac{16}{3}X^3+(8-2e^{2\pi it})X^2-\frac{1}{3}.
\end{array}    $$
\begin{lemma}\label{lema42}
The polynomials $\Gamma_1(t)$ and $\Gamma_2(t)$ are loops in $R\mathcal{C}_4$ at the 
base point, lifts of loops $\gamma_1(t)$ and $\gamma_2(t)$. 
\end{lemma}
\begin{proof}
It is clear that $\Gamma_k(0)=\Gamma_k(1)=X^4-\frac{16}{3}X^3+6X^2-\frac{1}{3}$ and  
$D(\Gamma_k(t))$ is $\gamma_k(t)$, $k=1,2$. We will use the notation $E$ for 
$\pm e^{\pi it}$.

The roots of $\gamma_1(t)$ are $3,\frac{1}{2}\pm\frac{1}{2}e^{\pi it}$. We have
$$ \Gamma_1(3)=\frac{9}{2}E^2-\frac{83}{6}\neq 0,\,\Gamma_1\left(\frac{1}{2}+\frac{1}{2}
E\right)=-\frac{1}{16}E^4+\frac{5}{6}E^3+\frac{11}{8}E^2-\frac{13}{16}\neq 0 $$
(the polynomial $Y^4-\frac{40}{3}Y^3-22Y^2+13$ has two real roots, 
$y_1\in(\frac{1}{2},\frac{2}{3})$, $y_2\in(14,15)$ and two conjugate complex roots
$y_3,y_4$; using the product of the roots we find that $|y_3|^2>\frac{13}{10}$, so
$E$ cannot be a root of this polynomial).

The roots of $\gamma_2(t)$ are $0,2\pm e^{\pi it}$. We have
$$ \Gamma_2(0)=-\frac{1}{3}\neq 0,\,\Gamma_2(2+E)=-E^4-\frac{16}{3}E^3-8E^2+5\neq 0 $$
(the polynomial $Y^4+\frac{16}{3}Y^3+8Y^2-5$ has two real roots, 
$y_1\in(-2,-1)$, $y_2\in(0,1)$ and two conjugate complex roots
$y_3,y_4$; using the sum of the roots we find that ${\rm Re}(y_3)<-\frac{5}{3}$, 
therefore $|y_3|\neq 1$ and $E$ cannot be a root of this polynomial).
\end{proof}
The fibration
$$  \mathbb{C}\setminus\left\lbrace-\frac{5}{3},0,9\right\rbrace\hookrightarrow 
   R\mathcal{C}_4\underset{\underset{I}{\leftarrow}}{\longrightarrow}   Q\mathcal{C}_3 $$
shows that $R\mathcal{B}_4\cong F\langle d_1,d_2,d_3\rangle\rtimes Q\mathcal{B}_3$. The 
generators of $R\mathcal{C}_4$ are given by the braids
$\delta_k(t)=X^4-\frac{16}{3}X^3+6X^2+d_k(t)$, $k=1,2,3$, together with the braids 
$\Gamma_1,\Gamma_2$. 

Now we begin to analyse the trajectories $X_1(t),X_2(t),X_3(t),X_4(t)$ of the roots 
of polynomials $\delta_*(t)$ and $\Gamma_*(t)$. These are $4\times 5$ continuous functions
starting at $X_1(0)=\frac{2-\sqrt{7}}{3}$, $X_2(0)=2-\sqrt{3}$, 
$X_3(0)=\frac{2+\sqrt{7}}{3}$ and $X_4(0)=2+\sqrt{3}$.
Take $P(X)=X^4-\frac{16}{3}X^3+6X^2$. The polynomial 
$\delta_k(t)=P(X)+\theta_k(t)+i\eta_k(t)$ could have a real root only if $d_k(t)$ is real, 
that is if $t\in[0,\frac{1}{3}]\cup\{\frac{1}{2}\}\cup[\frac{2}{3},1]$ for $k=1,2$ and if
$t\in[0,\frac{1}{12}]\cup[\frac{1}{4},\frac{1}{3}]\cup\{\frac{1}{2}\}\cup[\frac{2}{3},
\frac{3}{4}]\cup[\frac{11}{12},1]$ for $k=3$; Rolle theorem gives the intervals where 
these real roots lie. From the table (the free term of the polynomial $P(X)+\theta$ 
appears in the first column and bold ${\bf 0}$ stands for a double root) one can see how 
the real roots are moving on the real line: 
\begin{center}
\begin{tabular}{c|c|c|c|c|c|c|c|c|c|c|c}
$\theta$ & $\frac{5-\sqrt{40}}{3}$ & $\frac{2-\sqrt{7}}{3}$   & $0$  & $2-\sqrt{3}$  & $1$                      
         & $\frac{2+\sqrt{7}}{3}$  & $\frac{8-\sqrt{10}}{3}$  & $3$                      
         & $\frac{8+\sqrt{10}}{3}$ & $2+\sqrt{3}$             & $\frac{5+\sqrt{40}}{3}$   \\
\hline
\hline 
$0$ & $\frac{5}{3}$     & $\frac{1}{3}$         &  ${\bf 0}$       & $\frac{1}{3}$ 
    & $\frac{5}{3}$     & $\frac{1}{3}$         &  $0$             & $-9$    
    & $0$               & $\frac{1}{3}$         &  $\frac{5}{3}$                          \\
\hline 
$-\frac{1}{3}$
    & $+$ & $0$  & $-$  & $0$ & $+$      & $0$  & $-$ & $-$        & $-$ & $0$ & $+$      \\
\hline
\hline 
$-\frac{5}{3}+\varepsilon$
    & $+$ & $-$  & $-$  & $-$ & $+$       & $-$  & $-$ & $-$       & $-$ & $-$ & $+$      \\
\hline
$-\frac{5}{3}$
    & $0$ & $-$  & $-$  & $-$ & ${\bf 0}$ & $-$  & $-$ & $-$       & $-$ & $-$ & $0$      \\
\hline
$-\frac{5}{3}-\varepsilon$
    & $-$ & $-$  & $-$  & $-$ & $-$       & $-$  & $-$ & $-$       & $-$ & $-$ & $-$      \\
\hline
\hline
$-\varepsilon$
    & $+$ & $+$  & $-$  & $+$ & $+$       & $+$  & $-$ & $-$       & $-$ & $+$ & $+$      \\
\hline
$\varepsilon$
    & $+$ & $+$  & $+$  & $+$ & $+$       & $+$  & $+$ & $-$       & $+$ & $+$ & $+$      \\
\hline
\hline
$9-\varepsilon$
    & $+$ & $+$  & $+$  & $+$ & $+$       & $+$  & $+$ & $-$       & $+$ & $+$ & $+$      \\       
\hline
$9$
    & $+$ & $+$  & $+$  & $+$ & $+$       & $+$  & $+$ & ${\bf 0}$ & $+$ & $+$ & $+$      \\ 
\hline
$9+\varepsilon$
    & $+$ & $+$  & $+$  & $+$ & $+$       & $+$  & $+$ & $+$   & $+$ & $+$ & $+$          \\ 
\hline  
\hline            
\end{tabular}
\end{center}
The position of real roots of $\delta_k(\frac{1}{3})$ will be used to find the signs of $b$'s, 
the imaginary parts of the roots when these roots leave the real line; these are given by
$$ {\rm Im}(P(a+bi)+\theta+i\eta)=bP'(a)+b^2Q(a,b)+\eta=0,   $$
where only signs of $P'(a)=4a(a-1)(a-3)$ and of $\eta=\pm\varepsilon\sin(6\pi it)$ are
important (near the real roots of $\delta_k(\frac{1}{3})$ the term $bQ(a,b)$ is close to $0$).
For instance, look at the root $X_1(t)$ of $\delta_1(t)$: for $t\in[o,\frac{1}{3}]$, $X_1(t)$
is real, where $X_1(\frac{1}{3})=a\in(\frac{2-\sqrt{7}}{3},0)$. For this $a$ we have
$P'(a)<0$ and for $t=\frac{1}{3}+$ we have $\eta=-\varepsilon\sin(6\pi it)<0$, therefore
${\rm Im}(X_1(\frac{1}{3}+))<0$, so $X_1(t)$ enters the half plane ${\rm Im}(z)<0$.
The results for the rest of 11 roots are contained in the following three pictures. 
\begin{lemma}\label{lema43}
a) The polynomials $\delta_1(t)$ have no root on the lines ${\rm Re}(z)=1$ and 
${\rm Re}(z)=3$.

\noindent b) The polynomials $\delta_2(t)$ have no root on the lines ${\rm Re}(z)=0$ and 
${\rm Re}(z)=3$.

\noindent c) The polynomials $\delta_3(t)$ have no root on the line ${\rm Re}(z)=1$.
\end{lemma}
\begin{proof}
Real and imaginary parts of the equation $\delta_k(X)(a+bi)=0$ give
$$ \begin{array}{ccc}
a=0:      &      a=1:    &  a=3:                               \\
\left\lbrace  \begin{array}{l} b^4-6b^2+\theta=0               \\
                               \frac{16}{3}b^3+\eta=0     \end{array} \right. &
\left\lbrace  \begin{array}{l} b^4+4b^2+\frac{5}{3}+\theta=0  \\
                               \frac{4}{3}b^3+\eta=0      \end{array} \right. &                                    
\left\lbrace  \begin{array}{l} b^4-12b^2-9+\theta=0            \\
                               -\frac{20}{3}b^3+\eta=0   \end{array} \right.
\end{array}  $$
In all these three cases and for any $k=1,2,3$, we have $|\eta|\leq\varepsilon$, hence
$|b|$ is small.

a) For $k=1$ we have  $-\frac{1}{3}\leq\theta\leq\varepsilon$, and the first equation in 
the system for $a=1$ and $a=3$ cannot have small solutions.

b) For $k=2$ we have  $\theta\leq-\frac{1}{3}$, hence the solutions
of the first equation for $a=0$ and $a=3$ cannot be too small.

c) For $k=3$ we have $\theta\geq-\frac{1}{3}$, 
therefore the first equation for $a=1$ cannot have small solutions.
\end{proof}
Using these separating lines, the real roots of $\delta_k(t)$, the signs of imaginary parts 
of the complex roots, and also the symmetry $\delta_k(t)=\overline{\delta_k(1-t)}$ (this
implies that the roots of $\delta_k(t)$ for $t\in[\frac{1}{2},1]$ are the conjugates of the
roots of $\delta_k(1-t)$), we obtain the pictures
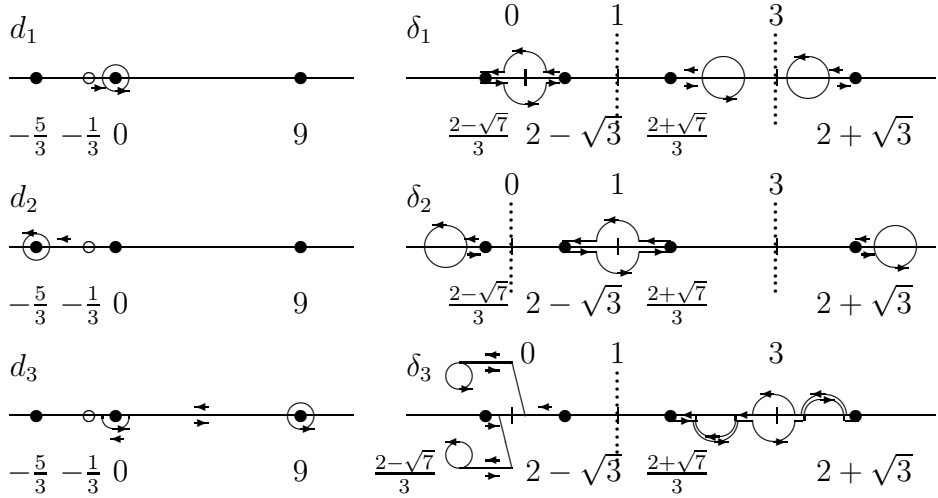
\begin{center}
\begin{picture}(360,60)        
\put(10,30){\line(1,0){130}}    \put(17,27){$\bullet $}       \put(47,27){$\bullet $} 
\put(117,27){$\bullet $}        \put(9,5){$-\frac{5}{3}$}     \put(29,5){$-\frac{1}{3}$} 
\put(49,5){$0$}                 \put(117,5){$9$}              \put(37,27){$\circ$}      
\put(160,30){\line(1,0){200}}   \put(10,45){$d_1$}            \put(160,45){$\delta_1$}
\put(187,27){$\bullet $}        \put(217,27){$\bullet $}      \put(257,27){$\bullet $} 
\put(327,27){$\bullet $}        \put(205,5){$2-\sqrt{3}$}     \put(41,26){\vector(1,0){5}}  
\put(315,5){$2+\sqrt{3}$}       \put(197,50){$0$}             \put(237,50){$1$}
\put(297,50){$3$}               \put(205,27){\line(0,1){6}}   \put(240,27){\line(0,1){6}}
\put(300,27){\line(0,1){6}}     \put(50,30){\circle{10}}      \put(50,25){\vector(1,0){5}} 
\multiput(188,28)(0,4){2}{\line(1,0){5}}   \put(205,32){\oval(16,16)[t]}  
\multiput(213,28)(0,4){2}{\line(1,0){5}}   \put(205,28){\oval(16,16)[b]}     
\put(205,20){\vector(1,0){5}}   \put(280,30){\circle{16}}       
\put(280,22){\vector(1,0){5}}   \put(50,25){\vector(1,0){5}}  \put(312,38){\vector(-1,0){5}}                     
\put(265,27){\vector(1,0){5}}   \put(270,33){\vector(-1,0){5}}\put(325,33){\vector(-1,0){5}} 
\put(205,40){\vector(-1,0){5}}  \put(323,27){\vector(1,0){5}} \put(193,28){\vector(1,0){5}}
\put(213,28){\vector(1,0){5}}   \put(197,32){\vector(-1,0){6}}\put(217,32){\vector(-1,0){5}}
\multiput(238,10)(0,3){12}{$\cdot$}        \put(175,5){$\frac{2-\sqrt{7}}{3}$}
\multiput(298,10)(0,3){12}{$\cdot$}        \put(250,5){$\frac{2+\sqrt{7}}{3}$}
\put(312,30){\circle{16}} 
\end{picture}
\end{center}
\begin{center}
\begin{picture}(360,60)        
\put(10,30){\line(1,0){130}}    \put(17,27){$\bullet $}       \put(47,27){$\bullet $} 
\put(117,27){$\bullet $}        \put(9,5){$-\frac{5}{3}$}     \put(29,5){$-\frac{1}{3}$} 
\put(49,5){$0$}                 \put(117,5){$9$}              \put(37,27){$\circ$}      
\put(160,30){\line(1,0){200}}   \put(10,45){$d_2$}            \put(160,45){$\delta_2$}
\put(187,27){$\bullet $}        \put(217,27){$\bullet $}      \put(257,27){$\bullet $} 
\put(327,27){$\bullet $}        \put(205,5){$2-\sqrt{3}$}     \put(20,30){\circle{10}}
\put(20,35){\vector(-1,0){5}}   \put(33,33){\vector(-1,0){5}} \put(175,30){\circle{16}} 
\put(175,38){\vector(-1,0){5}}  \put(187,33){\vector(-1,0){5}}\put(183,27){\vector(1,0){5}}
\put(240,32){\oval(16,16)[t]}   \put(240,28){\oval(16,16)[b]}    
\multiput(219,28)(0,4){2}{\line(1,0){13}}                     \put(227,32){\vector(-1,0){5}}  
\multiput(248,28)(0,4){2}{\line(1,0){12}}                     \put(255,32){\vector(-1,0){5}}
\put(240,20){\vector(1,0){5}}   \put(240,40){\vector(-1,0){5}}\put(225,28){\vector(1,0){5}}
\put(253,28){\vector(1,0){5}}   \put(335,33){\vector(-1,0){5}}\put(300,27){\line(0,1){6}} 
\put(345,30){\circle{16}}       \put(345,22){\vector(1,0){5}} \put(332,27){\vector(1,0){5}}
\put(315,5){$2+\sqrt{3}$}       \put(197,50){$0$}             \put(237,50){$1$}
\put(297,50){$3$}               \put(200,27){\line(0,1){6}}   \put(240,27){\line(0,1){6}}
\multiput(198,10)(0,3){12}{$\cdot$}        \put(175,5){$\frac{2-\sqrt{7}}{3}$}
\multiput(298,10)(0,3){12}{$\cdot$}        \put(250,5){$\frac{2+\sqrt{7}}{3}$}     
\end{picture}
\end{center}
\begin{center}
\begin{picture}(360,60)        
\put(10,30){\line(1,0){130}}    \put(17,27){$\bullet $}       \put(47,27){$\bullet $} 
\put(117,27){$\bullet $}        \put(9,5){$-\frac{5}{3}$}     \put(29,5){$-\frac{1}{3}$} 
\put(49,5){$0$}                 \put(117,5){$9$}              \put(37,27){$\circ$}      
\put(160,30){\line(1,0){200}}   \put(10,45){$d_3$}            \put(160,45){$\delta_3$}
\put(187,27){$\bullet $}        \put(217,27){$\bullet $}      \put(257,27){$\bullet $} 
\put(327,27){$\bullet $}        \put(147,5){$\frac{2-\sqrt{7}}{3}$} 
\put(205,5){$2-\sqrt{3}$}       \put(250,5){$\frac{2+\sqrt{7}}{3}$} 
\put(315,5){$2+\sqrt{3}$}       \put(203,50){$0$}             \put(237,50){$1$}
\put(297,50){$3$}               \put(200,27){\line(0,1){6}}   \put(240,27){\line(0,1){6}}
\put(300,27){\line(0,1){6}}     \put(50,30){\oval(10,10)[b]}  \put(277,30){\oval(15,16)[b]}
\put(277,28){\oval(17,16)[b]}   \put(260,28){\line(1,0){10}}  \put(285,28){\line(1,0){6}}
\put(318,30){\oval(17,16)[t]}   \put(318,28){\oval(15,16)[t]} \put(120,30){\circle{10}}     
\put(299,30){\oval(16,16)[t]}   \put(299,28){\oval(16,16)[b]} \put(307,28){\line(1,0){3}}
\put(195,30){\line(1,-4){5}}    \put(50,25){\vector(1,0){5}}  \put(53,21){\vector(-1,0){5}}  
\put(205,30){\line(-1,4){5}}    \put(80,27){\vector(1,0){5}}  \put(85,33){\vector(-1,0){5}}
\put(120,25){\vector(1,0){5}}   \put(277,20){\vector(1,0){5}} \put(278,22){\vector(-1,0){5}}
\put(299,20){\vector(1,0){5}}   \put(299,38){\vector(-1,0){5}}\put(318,38){\vector(-1,0){5}}
\put(317,36){\vector(1,0){5}}   \put(180,20){\vector(-1,0){5}}\put(180,40){\vector(1,0){5}}
\multiput(180,10)(0,40){2}{\line(1,0){20}}     \put(326,28){\line(1,0){5}}  
\multiput(238,10)(0,3){12}{$\cdot$}            \put(269,30){\vector(-1,0){5}}
\multiput(180,15)(0,30){2}{\circle{10}}        \multiput(195,13)(0,40){2}{\vector(-1,0){5}} 
\multiput(190,7)(0,40){2}{\vector(1,0){5}}                    \put(190,26){\vector(1,0){5}}
\put(215,33){\vector(-1,0){5}}  \put(263,28){\vector(1,0){5}} \put(290,30){\vector(-1,0){5}}
\end{picture}
\end{center}
Only in the pictures of $d_*$ we have circles and semicircles; for the pictures of $\delta_*$
the 'circles' and 'semicircles' are loops or curves, only their position relative to the 
upper half-plane and the lines Re$(z)=0,1,3$ is drawn correctly.
\begin{remark}
In the case of $\delta_3$ we have a lack of symmetry for $t\in[\frac{3}{4},\frac{11}{12}]$;
the roots of $\delta_3$ on this interval return on the trajectories of the roots for
$t\in[\frac{1}{12},\frac{1}{4}]$. 
\end{remark}
One can draw the braids $\delta_k(t)$ (in the following picture, the vertical threads 
correspond to real roots):
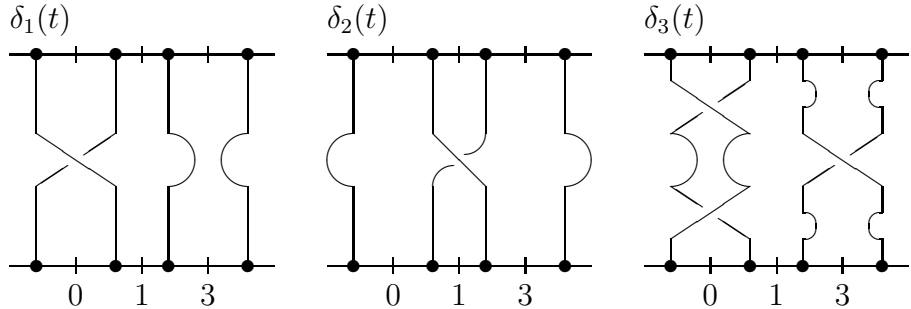
\begin{center}
\begin{picture}(360,120)      
\multiput(10,20)(120,0){3}{\multiput(0,0)(0,80){2}{\put(0,0){\line(1,0){100}}  
                     \multiput(25,-3)(25,0){3}{\line(0,1){6}}\put(7,-3){$\bullet $}
                     \put(37,-3){$\bullet $}\put(57,-3){$\bullet $}\put(87,-3){$\bullet $}} 
                     \put(22,-15){$0$} \put(47,-15){$1$} \put(72,-15){$3$}}  
\put(10,110){$\delta_1(t)$}    \put(130,110){$\delta_2(t)$}   \put(250,110){$\delta_3(t)$} 
\multiput(20,20)(120,0){2}{\multiput(0,0)(0,50){2}{\put(0,0){\line(0,1){30}}
         \put(30,0){\line(0,1){30}}\put(50,0){\line(0,1){30}}\put(80,0){\line(0,1){30}}}}
\put(20,70){\line(3,-2){30}}   \put(20,50){\line(3,2){12}}    \put(50,70){\line(-3,-2){12}}
\put(170,70){\line(1,-1){20}}  \put(260,30){\line(3,2){30}}   \put(260,50){\line(3,-2){12}}
\put(178,50){\oval(16,16)[tl]} \put(182,70){\oval(16,16)[br]} \put(290,30){\line(-3,2){12}}
\multiput(260,20)(0,70){2}{\put(0,0){\line(0,1){10}}
         \put(30,0){\line(0,1){10}}\put(50,0){\line(0,1){10}}\put(80,0){\line(0,1){10}}}
\multiput(70,60)(150,0){2}{\oval(20,20)[r]}    \multiput(100,60)(40,0){2}{\oval(20,20)[l]} 
\multiput(310,35)(0,50){2}{\oval(10,10)[r]}    \multiput(340,35)(0,50){2}{\oval(10,10)[l]}
\put(260,60){\oval(20,20)[r]}  \put(290,60){\oval(20,20)[l]}  \put(290,90){\line(-3,-2){12}} 
\put(260,70){\line(3,2){12}}   \put(310,50){\line(3,2){12}}   \put(340,70){\line(-3,-2){12}}
\multiput(310,40)(30,0){2}{\multiput(0,0)(0,30){2}{\put(0,0){\line(0,1){10}}}}
\multiput(260,90)(50,-20){2}{\line(3,-2){30}} 
\end{picture}
\end{center}
It is clear from this picture that $j_*(\delta_k(t))=x_k$, $k=1,2,3$.

The polynomials $\Gamma_k(t)$ could have real roots only for $t\in\{0,\frac{1}{2},1\}$;
for $t=\frac{1}{2}$ there are two real roots and two complex roots
$$ \begin{array}{ll}
   \Gamma_1(\frac{1}{2})=X^4-\frac{16}{3}X^3+7X^2-6X-\frac{1}{3}, & 
   x_1\in(\frac{2-\sqrt{7}}{3},0),x_2\in(2+\sqrt{3},4)                       \\
   \Gamma_2(\frac{1}{2})=X^4-\frac{16}{3}X^3+10X^2-\frac{1}{3},   & 
   x_1\in(\frac{2-\sqrt{7}}{3},0),x_2\in(0,2-\sqrt{3}).
\end{array} $$
The signs of imaginary parts of the roots of $\Gamma_k(t)$ near $t=0$ are given by the signs 
of coefficients of $b$ and the free terms in the following formulae ($R$ and $S$ are 
polynomials in $a,b$ and $\sin 2\pi t,\cos 2\pi t$)
$$ \begin{array}{l}
   {\rm Im}(\Gamma_1(a+bi))=b[4a(a^2-4a+3)+(1-\cos 2\pi t)(a-3)+bR]+
                                            \frac{\sin 2\pi t}{2}(6a-a^2)     \\
   {\rm Im}(\Gamma_2(a+bi))=b[4a(a^2-4a+4-\cos 2\pi t)+bS]-2a^2\sin 2\pi t.
\end{array} $$
\begin{lemma}\label{lema44}
a) The polynomials $\Gamma_1(t)$ have no root on the lines ${\rm Re}(z)=0$ and 
${\rm Re}(z)=3$.

\noindent b) The polynomials $\Gamma_2(t)$ have no root on the lines ${\rm Re}(z)=0$ and 
${\rm Re}(z)=1$.
\end{lemma}
\begin{proof}
a) If $3+bi$ is a root of $\Gamma_1(t)$, we find that 
$$ \left\lbrace  \begin{array}{l}
  b^4+(\frac{1}{2}\cos 2\pi t-\frac{25}{2})b^2+\frac{9}{2}\cos 2\pi t-\frac{83}{6}=0 \\
  \frac{20}{3}b^3-\frac{\sin 2\pi t}{2}b^2+\frac{9}{2}\sin 2\pi t=0     
                 \end{array} \right. . $$
The first equation gives $b^2\geq 12$, hence $|b|>3$. From 
$$ \left\vert\frac{20}{3}b-\frac{1}{2}\sin 2\pi t\right\vert b^2>10\cdot 9>
          \frac{9}{2}\vert\sin 2\pi t\vert  $$ 
we find a contradiction. If $bi$ is a root of $\Gamma_1(t)$, we have the system of equation, 
linear in $\sin$ and $\cos$:
$$ \left\lbrace  \begin{array}{l}
  b^2\cos 2\pi t-6b\sin 2\pi t=-2b^4+13b^2+\frac{2}{3}                   \\
  6\cos 2\pi t+b\sin 2\pi t=-\frac{32}{3}b^2+6.     
                 \end{array} \right.  $$
The relation $(\sin 2\pi t)^2+(\cos 2\pi t)^2=1$ gives
$$ (4b^4-216b^2-12)^2+(6b^5+153b^3-110b)^2-9(b^3+36b)^2=0,  $$
a polynomial in $b^2$ with positive coefficients.

\noindent b) If $bi$ is a root of $\Gamma_2(t)$, we find that
$$ b^4-(8-2\cos 2\pi t)b^2-\frac{1}{3}=0 \mbox{ and }
                                    \frac{16}{3}b^3+2b^2\sin 2\pi t=0. $$
The first equation gives $b^2\geq 6$ and the second equation gives $|b|\leq \frac{3}{8}$.
If $\Gamma_2(t)$ has a root $1+bi$, we find the system
$$ \left\lbrace  \begin{array}{l}
  (2b^2-2)\cos 2\pi t+4b\sin 2\pi t=-b^4-2b^2-\frac{10}{3}                \\
  -2b\cos 2\pi t+(b^2-1)\sin 2\pi t=-\frac{2}{3}b^3-2b.     
                 \end{array} \right.  $$
As in part a), we get 
$$ (10b^5+20b^3+8b)^2+(-3b^6+5b^4+20b^2+10)^2-36(b^2+1)^4=0,  $$
another polynomial in $b^2$ with positive coefficients.
\end{proof}
The roots of $\Gamma_*(t)$ and those of $\Gamma_*(1-t)$ are conjugate. Putting together 
all these facts, we obtain the picture
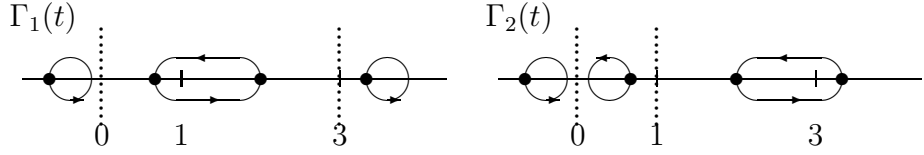
\begin{center}
\begin{picture}(360,60)        
\multiput(10,30)(180,0){2}{\put(0,0){\line(1,0){160}}\multiput(28,-20)(0,3){13}{$\cdot$} 
    \multiput(7,-3)(40,0){4}{$\bullet$}  \put(27,-25){$0$}  \put(57,-25){$1$} 
\put(117,-25){$3$}            \put(60,-3){\line(0,1){6}}    \put(120,-3){\line(0,1){6}}
\put(18,0){\oval(16,16)}      \put(18,-8){\vector(1,0){5}}}
\put(5,50){$\Gamma_1(t)$}     \put(185,50){$\Gamma_2(t)$}   \put(148,30){\oval(16,16)} 
\put(148,22){\vector(1,0){5}} \put(232,30){\oval(16,16)}    \put(232,38){\vector(-1,0){5}}    
\multiput(128,10)(0,3){13}{$\cdot$}           \multiput(248,10)(0,3){13}{$\cdot$}
\multiput(80,30)(220,0){2}{\put(0,0){\oval(40,16)}
\put(0,-8){\vector(1,0){5}}   \put(0,8){\vector(-1,0){5}}}
\end{picture}
\end{center}
and the braids 
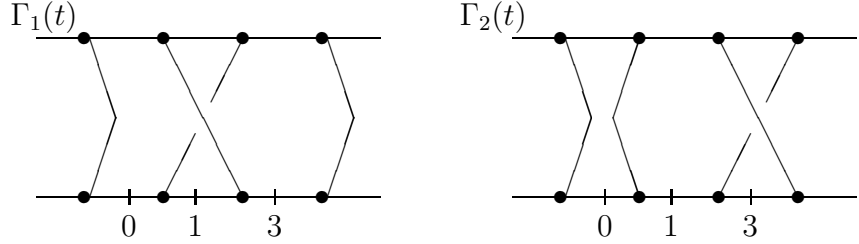
\begin{center}
\begin{picture}(360,100)        
\multiput(30,20)(180,0){2}{\multiput(0,0)(0,60){2}{\line(1,0){130}
\multiput(-115,-3)(30,0){4}{$\bullet$}}\put(32,-15){$0$} \put(57,-15){$1$}\put(87,-15){$3$}
\put(35,-3){\line(0,1){6}}     \put(60,-3){\line(0,1){6}}  \put(90,-3){\line(0,1){6}} }
\put(20,85){$\Gamma_1(t)$}     \put(190,85){$\Gamma_2(t)$}
\multiput(50,20)(90,0){2}{\put(0,0){\line(1,3){10}}         \put(0,60){\line(1,-3){10}}}  
\put(248,50){\line(1,3){10}}   \put(248,50){\line(1,-3){10}}
\multiput(78,20)(210,0){2}{\put(0,0){\line(1,2){12}}        \put(0,60){\line(1,-2){30}}
\put(30,60){\line(-1,-2){12}}} \put(230,20){\line(1,3){10}} \put(230,80){\line(1,-3){10}}
\end{picture}
\end{center}
As a consequence we find that $j_*(\Gamma_k)=x_{k+1}$ for $k=1,2$.


\section{Complements: higher dimensions and real configurations}
 
First we have a look at $Q\mathcal{F}_n$ for $n\geq 4$. 
The space $Q\mathcal{F}_4\subseteq\mathbb{C}^4$ is the complement of the arrangement 
with six diagonal hyperplanes $H_{ij}$ and six hyperquadrics $S_{ij}^{(4)}$
$$ H_{ij}:z_i=z_j,\, S_{ij}^{(4)}:3z_i^2+4z_iz_j+3z_j^2-5(z_h+z_k)(z_i+z_j)+10z_hz_k=0. $$
Every hyperquadric $S_{ij}^{(4)}$ has a line of critical points $z_1=z_2=z_3=z_4$. 
The space $Q\mathcal{F}_5\subseteq\mathbb{C}^5$ is the complement of the arrangement 
with ten diagonal hyperplanes $H_{ij}$ and ten hypercubics $S_{ij}^{(5)}$
$$ \begin{array}{ll}
S_{ij}^{(5)}: & 4z_i^3+6z_i^2z_j+6z_iz_j^2+4z_j^3-2(z_h+z_k+z_l)(3z_i^2+4z_iz_j+3z_j^2)+ \\
              & +10(z_hz_k+z_hz_l+z_kz_l)(z_i+z_j)-20z_hz_kz_l=0. 
\end{array}      $$
More singularities are here: $S_{ij}^{(5)}$ contains 3 two-planes of singular points:
$$ P_{ijhk}:z_i=z_j=z_h=z_k,\, P_{ijhl}:z_i=z_j=z_h=z_l,\, P_{ijkl}:z_i=z_j=z_k=z_l. $$

In general $Q\mathcal{F}_n\subseteq\mathbb{C}^n$ is the complement of the arrangement 
with $n\choose 2$ diagonal hyperplanes $H_{ij}$ and $n\choose 2$ hypersurfaces 
$S_{ij}^{(n)}$ given by homogeneous polynomials of degree $n-2$ ($S_{ij}$ is symmetric 
in variables $z_i,z_j$ and linear in the elementary symmetric polynomials of the rest of 
$n-2$ variables).

For the proof of Proposition \ref{propS} we need the following Lemma:
\begin{lemma}\label{lemaS}
Let $U$ be a connected open dense subset and $H$ an algebraic hypersurface in $\mathbb{C}^n$.
Then $U\setminus H$ is a connected open dense subset in $\mathbb{C}^n$ and
$$ i_*:\pi_1(U\setminus H)\twoheadrightarrow\pi_1(U) $$
is a surjective homomorphism.

In particular 
$$ \tilde{\textit{\j}}_*:Q\mathcal{P}_{n-1}=\pi_1(Q\mathcal{F}_{n-1})
        \twoheadrightarrow\pi_1(\mathcal{F}_{n-1})=\mathcal{P}_{n-1}  $$
is a surjective homomorphism.
\end{lemma}
\begin{proof}
If $\alpha:[0,1]\rightarrow U$ is a smooth path (or a smooth loop), one can find a small
deformation $\beta$, a smooth path (or a smooth loop), transversal to all the strata of
$Sing(H)$, the singular locus of $H$, and also transversal to $H\setminus Sing(H)$.
\end{proof}
As a consequence, from the diagram
$$ \begin{array}{ccccccccc} 
1 & \rightarrow & Q\mathcal{P}_{n-1} & \rightarrow  & Q\mathcal{B}_{n-1} & \rightarrow &
                                       \Sigma_{n-1} &   \rightarrow      & 1            \\
  &             & \downarrow         &              &  \downarrow        &             &
                                      \updownarrow  &                    &              \\
1 & \rightarrow & \mathcal{P}_{n-1}  & \rightarrow  &  \mathcal{B}_{n-1} & \rightarrow &
                                       \Sigma_{n-1} &   \rightarrow      & 1                                       
\end{array} $$
the homomorphism $j_*:Q\mathcal{B}_{n-1}\twoheadrightarrow\mathcal{B}_{n-1}$ is surjective, too.

\noindent\textit{Proof of Proposition \ref{propS}.} The argument for the surjectivity of the 
homomorphisms
$$ \tilde{\textit{\j}}_*:R\mathcal{P}_n\twoheadrightarrow\mathcal{P}_n\mbox{ and }    
   j_*:R\mathcal{B}_n\twoheadrightarrow\mathcal{B}_n $$
is similar: the roots of $Q(X)$, the derivative of the polynomial 
$P(X)=\prod_{i=1}^n(X-\alpha_i)$ are (locally) holomorphic functions in 
$\alpha_1,\ldots,\alpha_n$ (because the roots $\{\beta_i\}$ are distinct) and the equations
$$ \alpha_j=\alpha_k \mbox{ and }\int_{\beta_k}^{\beta_j}Q(t)dt=0 $$
are given locally by holomorphic equations.   \hfill $\square$

Secondly, we analyse the derivative as a fibration in the case of real configuration spaces. It is
obvious that the covering $\mathcal{F}_n(\mathbb{R})\rightarrow\mathcal{C}_n(\mathbb{R})$ is
completely trivial:
$$ \begin{array}{ll}
   \mathcal{F}_n(\mathbb{R})\approx\mathcal{C}_n(\mathbb{R})\times\Sigma_n\stackrel{pr_1}
                         {\longrightarrow}\mathcal{C}_n(\mathbb{R})=
                         \{(x_1,x_2,\ldots,x_n)\in\mathbb{R}^n\mid x_1<x_2<\ldots<x_n\}, \\
   \mathcal{C}_n(\mathbb{R})\approx\mathbb{R}^n,\,(x_1,x_2,\ldots,x_n)
   \mapsto(x_1,\ln(x_2-x_1),\ldots,\ln(x_n-x_{n-1})).  
\end{array}   $$ 
Now we define the min-max $m(Q)$ and max-min $M(Q)$ of $Q(X)$, an $(n-1)$-degree polynomial 
with real distinct roots
($n\geq 3$).
\begin{definition}\label{def10}
If $Q(X)=n(X-b_1)(X-b_2)\ldots(X-b_{n-1})$, where $b_1<\ldots<b_{n-1}$, and
$P(X)=\int_0^XQ(t)dt$, then
$$ \begin{array}{l}
m(Q)={\rm max}\,P(b_k),\, k=n-1,n-3,\ldots \mbox{ and }       \\
    M(Q)={\rm min}\,P(b_k),\, k=n-2,n-4,\ldots   
\end{array} $$
\end{definition}
\begin{remark}
For $n=2$ and $Q(X)=2(X-b)$ (here $P(X)=X^2-2bX$), we define $m(Q)=-b^2$ and $M(Q)=\infty$. 
\end{remark}
The version with real points of the fibration given by derivative is contained in the 
following Theorem: here the new restricted configuration space is 
$R\mathcal{C}_n(\mathbb{R})=\mathcal{C}_n(\mathbb{R})=\mathcal{C}_n\cap\mathbb{R}^n$ and
the new restricted base configuration space, 
$Q\mathcal{C}_{n-1}(\mathbb{R})$, is defined in the statement of the Theorem.
\begin{theorem}\label{thm2}
a) For any $n\geq 2$ and any monic polynomial $P(X)\in\mathcal{C}_n(\mathbb{R})$, its 
derivative $D(P(X))=Q(X)$ has $n-1$ distinct real roots and $Q(X)$ verifies the inequality 
$m(Q)<M(Q)$.

\noindent b) The image of derivative 
$D:\mathcal{C}_n(\mathbb{R})\rightarrow \mathcal{C}_{n-1}(\mathbb{R})$ is the set
$$ \begin{array}{lll}
Q\mathcal{C}_{n-1}(\mathbb{R}) & = & \{P'(X)\mid P(X)=
           \prod_{i=1}^n(X-a_i),a_1<a_2<\ldots<a_n \}=                        \\
& = & \{Q(X)=n\prod_{i=1}^{n-1}(X-b_i),b_1<b_2<
                               \ldots<b_{n-1}\mid m(Q)<M(Q) \},  
\end{array}   $$
an open subset of $\mathcal{C}_{n-1}(\mathbb{R})$. We have 
$Q\mathcal{C}_{n-1}(\mathbb{R})=\mathcal{C}_{n-1}(\mathbb{R})$ if and only if $n=1,2,3$. 

\noindent c) For any $n\geq 2$ there is a homeomorphism 
$$ \mathcal{C}_n(\mathbb{R})\stackrel{D\times Ev_0}{\longrightarrow}
     Q\mathcal{C}_{n-1}(\mathbb{R})\times (0,1),  $$
where
$$ Ev_0(P(X))=\begin{cases} 
   \frac{P(0)+m(P')}{P(0)+m(P')-1}   & \mbox{ if }n=2 \mbox{ and }    \\
   \frac{P(0)+M(P'(X))}{M(P')-m(P')} & \mbox{ if }n\geq 3.             \\
              \end{cases}   $$     
\end{theorem}
\begin{proof}
a) For the polynomial $P(X)=\prod_{i=1}^n(X-a_i)$, where $a_1<a_2<\ldots<a_n$, its 
derivative $P'(X)$ has $n-1$ real roots $b_i\in(a_i,a_{i+1})$, $i=1,2,\ldots,n-1$ and also
$$ P(b_{n-1})<0,\,P(b_{n-2})>0,\,P(b_{n-3})<0,\,\ldots ,\,(-1)^{n-1}P(b_1)>0,  $$
therefore $m(P')<-P(0)<M(P')$.

\noindent b) For any $n\geq 2$, if 
$Q(X)=n\prod_{i=1}^{n-1}(X-b_i)$, with $b_1<b_2<\ldots<b_{n-1}$, 
satisfies $m(Q)<M(Q)$, then $P(X)=\int_0^XQ(t)dt-c$ with $c\in(m(Q),M(Q))$ has $n$
real distinct roots (from inequalities
$$ P(b_{n-1})<0,\,P(b_{n-2})>0,\,P(b_{n-3})<0,\,\ldots,\,(-1)^{n-1}P(b_1)>0,   $$
we find real roots $a_n\in(b_{n-1},\infty)$, $a_{n-1}\in(b_{n-2},b_{n-1})$,
$a_{n-2}\in(b_{n-3},b_{n-2}),\ldots$, $a_2\in(b_1,b_2)$, $a_1\in(-\infty,b_1)$);
this shows that $Q(X)$ is in the image of $D$. It is obvious that the fiber of such 
polynomial is
$$ D^{-1}(Q(X))=\left\lbrace\int_0^XQ(t)dt-c\mid c\in(m(Q),M(Q))\right\rbrace.   $$
For a polynomial $P(X)$ of degree $n=2,3,4$, the picture shows that always we 
have $m(P')<M(P')$:
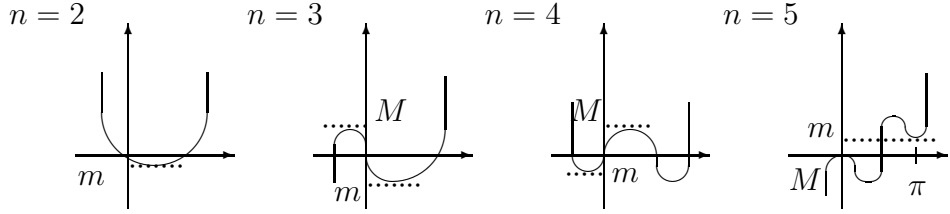
\begin{center}
\begin{picture}(360,90)        
\multiput(20,30)(90,0){4}{\put(0,0){\vector(1,0){60}}\put(20,-20){\vector(0,1){70}} }
\put(-5,80){$n=2$} \put(85,80){$n=3$} \put(175,80){$n=4$}            \put(265,80){$n=5$}
\put(124,20){\oval(12,40)[tl]}        \put(124,30){\oval(12,20)[tr]} \put(21,19){$m$}
\put(140,30){\oval(20,20)[bl]}        \put(140,60){\oval(40,80)[br]} \put(118,13){$m$}
\put(214,50){\oval(12,52)[bl]}        \put(214,30){\oval(12,12)[br]} \put(223,20){$m$}
\put(230,30){\oval(20,20)[t]}         \put(246,30){\oval(12,20)[bl]} \put(297,37){$m$}
\put(246,50){\oval(12,60)[br]}        \put(310,15){\oval(12,30)[tl]} \put(290,17){$M$}
\put(310,25){\oval(10,10)[tr]}        \put(320,25){\oval(10,10)[b]}  \put(207,43){$M$}
\put(330,25){\oval(10,40)[tl]}        \put(330,41){\oval(8,8)[tr]}   \put(133,43){$M$}
\put(338,41){\oval(8,8)[bl]}          \put(338,61){\oval(8,48)[br]}  \put(335,15){$\pi$}
\put(50,61.5){\oval(40,70)[b]}        \put(338,27){\line(0,1){6}}
\multiput(310,33)(3,0){12}{$\cdot$}   \multiput(220,38)(3,0){6}{$\cdot$}
\multiput(205,20)(3,0){5}{$\cdot$}    \multiput(130,16)(3,0){7}{$\cdot$}
\multiput(128,38)(-3,0){6}{$\cdot$}   \multiput(40,23)(3,0){7}{$\cdot$}
\end{picture}
\end{center}
The graph of $P_5(X)=X^5-(\frac{5}{4}\pi+5)X^4+(\frac{20}{3}\pi+5)X^3-\frac{15}{2}\pi X^2$ 
is given in the picture ($n=5$). We have 
$$ P'_5(X)=5X(X-1)(X-3)(X-\pi),\,P(0)=0,\,P(\pi)>\frac{\pi^3}{4}, $$
hence $m(P')=P(\pi)>P(0)=M(P')$.
By induction ($n\geq 5$), take $P_n(X)$ an $n$-degree monic polynomial with
$$ P'_n(X)=n(X-b_{n-1})(X-b_{n-2})\ldots(X-b_5)X(X-1)(X-3)(X-\pi), $$
$b_{n-1}<b_{n-2}<\ldots<b_5<0$
and $P_n(0)=0$, $P_n(\pi)>0$, hence $m(P'_n)>M(P'_n)$. Take 
$$ Q_n(X)=\frac{n+1}{n}(X-b_n)P'_n(X). $$ Then define
$$ P_{n+1}(X)=\int_0^XQ_n(t)dt=\frac{n+1}{n}
                                     \left[(X-b_n)P_n(X)-\int_0^XP_n(t)dt\right].  $$
Obviously $P_{n+1}(0)=0$ and, if we choose 
$$ b_n<{\rm min}\left(b_{n-1},\pi-P_n(\pi)^{-1}\int_0^{\pi}P_n(t)dt\right), $$
we obtain $m(P'_{n+1})\geq P_{n+1}(\pi)>P_{n+1}(0)\geq M(P'_{n+1})$. 

It is clear that, for $n\geq 5$, $Q\mathcal{C}_{n-1}(\mathbb{R})$ is not a dense 
subset of $\mathcal{C}_{n-1}(\mathbb{R})$.

\noindent c) The map $D\times Ev_0$ is continuous and its inverse is the continuous map
$$ (Q(X),c)\mapsto\begin{cases} 
   \int_0^XQ(t)dt+\frac{c}{c-1}-m(Q) & \mbox{ if }n=2 \mbox{ and }    \\
   \int_0^XQ(t)dt+[M(Q)-m(Q)]c-M(Q)  & \mbox{ if }n\geq 3.            \\
              \end{cases}   $$     
\end{proof}
\begin{remark}
For $n=2$ we have $Q\mathcal{C}_2(\mathbb{R})=\mathbb{R}\cap Q\mathcal{C}_2$,
but there is no such relation for higher degrees: 

a) the polynomial $Q_3(X)=4(X+5)X(X-5)$ belongs to 
$Q\mathcal{C}_3(\mathbb{R})\setminus Q\mathcal{C}_3$ (because $P_4(-5)$ and $P_4(5)$ are equal);

b) $Q_5(X)=5X(X-1)(X-3)(X-\pi)$ belongs to 
$\mathbb{R}^5\cap Q\mathcal{C}_5\setminus Q\mathcal{C}_5(\mathbb{R})$ (the values of
$P_6(X)=\int_0^XQ_5(t)dt$ at $0,1,3$ and $\pi$ are distinct).
\end{remark}


\mbox{ }


\begin{thebibliography}{00}

\bibitem{A} E.~Artin, {\em Theory of Braids}, Ann. of Math. \textbf{48}, 101-126 (1947).

\bibitem{B} J.~Birman, {\em Braids, Links, and Mapping Class Groups}, Princeton 
University Press (1974).


\bibitem{FN} E.~Fadell, L.~Neuwirth, {\em Configuration spaces}, Math. Scand. \textbf{10}, 
111-118 (1962).

\bibitem{MKS} W.~Magnus, A.~Karrass, D.~Solitar, {\em Combinatorial Group Theory}, 
Dover Publications, New York (1976).

\bibitem{M} S.~Moran, {\em The Mathematical Theory of Knots and Braids}, North Holland 
Mathematics Stu-dies, vol \textbf{82}, Elsevier Science Publishers B.V. (1983).

\bibitem{OT} P.~Orlik, H.~Terao, {\em Arrangements of Hyperplanes}, Grundlehren der 
mathematischen Wissenschaften \textbf{300}, Springer-Verlag (1992).

\bibitem{Pa} L.~Paris, {\em Braid groups and Artin groups}, arXiv:0711.2372.v1[GR]
(2007).

\bibitem{P} V.~V.~Prasolov, {\em Polynomials}, Algorithms and Computation in 
Mathematics, vol \textbf{11}, Springer-Verlag (2004).

\bibitem{R} S.~Roman, {\em Field Theory}, GTM 158, Springer-Verlag (2006).

\end{thebibliography}
\end{document}